\documentclass{amsart}

\usepackage{comment}
\usepackage{fullpage}
\usepackage{amssymb}
\usepackage[pdfauthor={Hiraku Abe, Lauren DeDieu, Federico Galetto, and Megumi Harada},pdftitle={Geometry of Hessenberg varieties with applications to Newton-Okounkov bodies}]{hyperref}
\usepackage{enumitem}
\usepackage{color}
\usepackage{tikz}
\usetikzlibrary{arrows}
\usepackage{mathdots}
\usepackage[smalltableaux,centertableaux]{ytableau}

\theoremstyle{definition}
\newtheorem{theorem}{Theorem}[section]
\newtheorem{proposition}[theorem]{Proposition}
\newtheorem{corollary}[theorem]{Corollary}
\newtheorem{lemma}[theorem]{Lemma}
\newtheorem{definition}[theorem]{Definition}
\newtheorem{remark}[theorem]{Remark}
\newtheorem{example}[theorem]{Example}

\numberwithin{table}{section}
\numberwithin{figure}{section}
\numberwithin{equation}{section}

\DeclareMathOperator{\GL}{GL}
\DeclareMathOperator{\Flags}{Flags}
\DeclareMathOperator{\Spec}{Spec}
\DeclareMathOperator{\Hess}{Hess}
\DeclareMathOperator{\Pet}{Pet}
\DeclareMathOperator{\codim}{codim}
\DeclareMathOperator{\id}{id}
\DeclareMathOperator{\pt}{pt}

\DeclareMathOperator{\rank}{rank}
\DeclareMathOperator{\Vol}{Vol}

\newcommand{\Sn}{\mathfrak{S}_n}
\renewcommand{\geq}{\geqslant}
\renewcommand{\leq}{\leqslant}
\newcommand{\CC}{\mathbb{C}}

\newcommand{\ZZ}{\mathbb{Z}}
\newcommand{\QQ}{\mathbb{Q}}
\newcommand{\A}{\mathbb{A}}
\newcommand{\X}{\mathfrak{X}}
\newcommand{\fid}[1][w]{\mathcal{J}_{#1,h}}
\renewcommand{\P}{\mathbb{P}}
\newcommand{\R}{\mathbb{R}}
\newcommand{\neib}[1]{\mathcal{N}_{#1}}
\newcommand{\dfv}{D}
\newcommand{\into}{\hookrightarrow}

\begin{document}

\title[Geometry of Hessenberg varieties]{Geometry of Hessenberg varieties with applications
  to Newton-Okounkov bodies}

\author{Hiraku Abe}
\address{Department of Mathematics and
Statistics\\ McMaster University\\ 1280 Main Street West\\ Hamilton, Ontario L8S4K1\\ Canada / Osaka City University Advanced Mathematical Institute\\ 3-3-138 Sugimoto, Sumiyoshi-ku\\ Osaka 558-8585\\ JAPAN}
\email{hirakuabe@globe.ocn.ne.jp} 

\author{Lauren DeDieu}
\address{Department of Mathematics and Statistics \\
University of Calgary \\
2500 University Drive NW \\
Calgary, Alberta
T2N 1N4 \\ Canada}
\email{lauren.dedieu@ucalgary.ca}

\author{Federico Galetto}
\address{Department of Mathematics and
Statistics\\ McMaster University\\ 1280 Main Street West\\ Hamilton, Ontario L8S4K1\\ Canada}
\email{galettof@math.mcmaster.ca} 
\urladdr{\url{http://math.galetto.org}}

\author{Megumi Harada}
\address{Department of Mathematics and
Statistics\\ McMaster University\\ 1280 Main Street West\\ Hamilton, Ontario L8S4K1\\ Canada}
\email{Megumi.Harada@math.mcmaster.ca}
\urladdr{\url{http://www.math.mcmaster.ca/Megumi.Harada/}}


\date{\today}

\begin{abstract}
  In this paper, we study the geometry of various Hessenberg varieties
  in type A, as well as families thereof. Our main results are as follows. We
  find explicit and computationally convenient generators for the
  local defining ideals of indecomposable regular nilpotent Hessenberg
  varieties, allowing us to conclude that all regular nilpotent Hessenberg
  varieties are local complete intersections.  We also show that
  certain flat families of Hessenberg varieties, whose generic fibers are
  regular semisimple Hessenberg varieties and whose special fiber is a
  regular nilpotent Hessenberg variety, have reduced
  fibres. In the second half of the paper we present several applications 
  of these results. First, we construct certain flags of subvarieties of a regular nilpotent
  Hessenberg variety, obtained by intersecting with Schubert
  varieties, with well-behaved geometric properties. Second, we give a computationally
  effective formula for the degree of a regular nilpotent Hessenberg
  variety with respect to a Pl\"ucker embedding.  Third, we explicitly compute 
  some Newton-Okounkov bodies of the two-dimensional Peterson variety. 
  \end{abstract}

\maketitle

\setcounter{tocdepth}{1}
\tableofcontents

\section{Introduction}

In this paper we study Hessenberg varieties of various types and 
families thereof, with a view towards applications and connections to other 
areas. Throughout this paper, for simplicity we restrict to Lie type A 
although we suspect that our discussion generalizes to other Lie
types.

Hessenberg varieties in type A are subvarieties of the full flag
variety $\Flags(\CC^n)$ of nested sequences of linear subspaces in
$\CC^n$. Their geometry and (equivariant) topology have been studied
extensively since the late 1980s \cite{DeM, DeMShay, DMPS}.  This
subject lies at the intersection of, and makes connections between,
many research areas such as geometric representation theory (see for example
\cite{Springer76, Fung03}), combinatorics (see e.g.\ \cite{Fulm, mb, sh-wa11, GuayPaquet, AHMMS}), and
algebraic geometry and topology (see e.g.\ \cite{Kostant, br-ca04, ty,
  IY, precup13a, precup13b, Abe-Crooks, br-ch}). 
A special case of Hessenberg varieties called the Peterson
variety $\Pet_n$ arises in the study of the quantum
cohomology of the flag variety \cite{Kostant, Rietsch}, and more
generally, geometric properties and invariants of many different types
of Hessenberg varieties (including in Lie types other than A) have
been widely studied. 

We now describe the main results of this paper. (For definitions we refer to
Section~\ref{sec:prelim-Hess}.)

\begin{enumerate}
\item We determine an explicit list of
  generators for the local defining ideals of indecomposable regular nilpotent
  Hessenberg varieties (Proposition~\ref{prop:defining ideal of Hessenberg}). 
\item We prove that certain flat families of Hessenberg varieties over $\A^1$ (or $\P^1$) have reduced fibers (Theorem~\ref{thm:3}). 
\end{enumerate} 
In the second part of the paper we give applications of the above. We were motivated from the theory of Newton-Okounkov bodies, but items (3) and (4) are also of independent interest. 
\begin{enumerate} 
\item[(3)] We construct families of flags $Y_{\bullet} = \{Y_0 = \Hess(N,h)
  \supset Y_1 \supset \cdots \supset Y_n\}$ of subvarieties in
  regular nilpotent Hessenberg varieties arising from intersections
  with (dual) Schubert
  varieties; the intersections are smooth at $Y_n = \{\pt\}$, where
  $n=\dim_{\CC} \Hess(N,h)$ (Theorem~\ref{theorem:flag of
    Schuberts}). 
\item[(4)] We give a computationally efficient formula for the degree of 
an arbitrary indecomposable regular nilpotent Hessenberg 
variety with respect to a Pl\"ucker embedding associated to a weight
$\lambda=(\lambda_1,\lambda_2,\cdots,\lambda_n)$ as a polynomial in
the $\lambda_i$ (Theorem~\ref{theorem: formula for the degree of nilpotent}). 
\item[(5)] We explicitly compute some 
  Newton-Okounkov bodies associated to the Peterson variety in
  $\Flags(\CC^3)$, a special case of regular nilpotent Hessenberg
  varieties (Theorem~\ref{prop:a2>a1}). 
\end{enumerate}
Some remarks are in order.  Firstly, our results in (1)
generalize a result of Insko and Yong \cite{IY} for the case of Peterson
varieties, and also a result of Insko \cite{Insko} showing that
regular nilpotent Hessenberg varieties are local complete
intersections 
when the Hessenberg function is strictly
increasing. Secondly, the family we consider in
(2) is presumably the one hinted at in
\cite[Remark 7.3]{AT}.
Thirdly, one 
reason for studying the flags of subvarieties in (4) is that
well-behaved such flags are often a crucial ingredient in the
construction of Newton-Okounkov bodies.  Fourthly, the polynomial mentioned in (5)
is called a \emph{volume polynomial} in \cite{AHMMS}, where the
authors also show that the natural Poincar\'e duality algebra
associated to this polynomial is in fact isomorphic to the ordinary
cohomology ring of the regular nilpotent Hessenberg variety. 
Finally, we view the results of (5) as a
first case of a Newton-Okounkov-type computation for
Hessenberg varieties. 

\bigskip

\noindent \textbf{Acknowledgements.} 
We are grateful to Mikiya Masuda for his stimulating questions and his
support and encouragement. We also thank Allen Knutson for 
pointing out to us the significance of the flat family of
regular Hessenberg varieties over the space of regular matrices, and to Sergio Da Silva for pointing 
out an error in a previous version of this manuscript. 
The first
author was supported in part by the JSPS Program for Advancing
Strategic International Networks to Accelerate the Circulation of
Talented Researchers: ``Mathematical Science of Symmetry, Topology and Moduli, Evolution of International Research Network based on
OCAMI.'' He is also supported by a JSPS Research Fellowship for Young Scientists Postdoctoral Fellow: 16J04761.  The fourth author was supported
by an NSERC Discovery Grant and a Canada Research Chair (Tier 2) Award.
The results of 
Section~\ref{sec:NOBY Peterson} are part of the second author's 
Ph.D.\ thesis \cite{DeDieu}. 
Part of the research for this paper was carried out at the Fields Institute; the authors would like to thank the institute for its hospitality. Finally, we thank the anonymous referee for 
constructive and concrete suggestions which substantially improved the manuscript.

\section{Background on Hessenberg varieties}\label{sec:prelim-Hess}

In this section we recall some basic definitions used in the study of
Hessenberg varieties. Since detailed exposition is available in the
literature \cite{ty, DMPS} we keep the discussion brief. 

By the \textbf{flag variety} we mean the homogeneous space
$\GL_n(\CC)/B$, where $B$ denotes the subgroup of upper-triangular
matrices. This homogeneous space may also be identified with the space
of nested sequences of linear subspaces of $\CC^n$, i.e.\ 
\begin{align}\label{eq:flag}
  \Flags(\CC^n) &:= 
  \{ V_{\bullet} = (\{0\} \subseteq V_1 \subseteq V_2 \subseteq \cdots V_{n-1} \subseteq
  V_n = \CC^n) \mid \dim_{\CC}(V_i) = i \} \\ \notag
  &\textcolor{white}{:}\cong \GL_n(\CC)/B;
\end{align}
the identification with $\GL_n(\CC)/B$ takes a coset $MB$, for $M \in \GL_n(\CC)$, to the
flag $V_\bullet$ with $V_i$ defined as the span of the leftmost $i$
columns of $M$.

We use the notation $[n] := \{1,2,\ldots,n\}$ throughout. 
A \textbf{Hessenberg function} is a function $h\colon [n] \to [n]$
satisfying $h(i) \geq i$ for all $1 \leq i \leq n$ and
$h(i+1) \geq h(i)$ for all $1 \leq i < n$. We frequently denote a
Hessenberg function by listing its values in sequence,
$h = (h(1), h(2), \ldots, h(n)=n)$. To a Hessenberg function
$h$ we associate a subspace of
$\mathfrak{gl}_n(\CC)$ (the vector space of $n \times n$ complex
matrices) defined as
\begin{equation}\label{eq:def of H(h)}
  H(h) := \{ (a_{i,j})_{i,j\in[n]} \in \mathfrak{gl}_n(\CC) \mid a_{i,j}=0 \text{ if $i>h(j)$} \},
\end{equation}
which we call the \textbf{Hessenberg subspace} $H(h)$.
It is sometime useful to visualize this space as a configuration of
boxes on a square grid of size $n\times n$ whose shaded boxes
correspond to the $a_{i,j}$ which are allowed to be non-zero (see Figure
\ref{fig:Hess_space}).

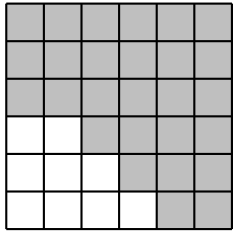
\begin{figure}[h]
  \centering
  \begin{tikzpicture}[scale=0.5]
    \fill[lightgray] (0,6)--(0,3)--(2,3)--(2,2)--(3,2)--(3,1)
    --(4,1)--(4,0)--(6,0)--(6,6)--cycle;
  \foreach \i in {0,...,6} {
    \draw[thick] (\i,0)--(\i,6);
    \draw[thick] (0,\i)--(6,\i);
    }
  \end{tikzpicture}

\caption{The picture of $H(h)$ for $h=(3,3,4,5,6,6)$.}
\label{fig:Hess_space}
\end{figure}

We can now define the central object of study. 

\begin{definition}\label{HessDefn}
  Let $A\colon \CC^n\rightarrow\CC^n$ be a linear operator and
  $h\colon [n] \to [n]$ a Hessenberg function. 
  The \textbf{Hessenberg variety} associated to $A$ and
  $h$ is defined to be $$\Hess(A,h):=\{V_\bullet\in \Flags(\CC^n)\mid
  AV_i\subseteq V_{h(i)}, \ \forall i\}.$$ Equivalently, under the
  identification~\eqref{eq:flag} and viewing $A$ as an element in $\mathfrak{gl}_n(\CC)$, 
\begin{equation}\label{eq:def Hess conj}
\Hess(A,h) = \{ MB \in \GL_n(\CC)/B \mid M^{-1}AM \in H(h) \}.
\end{equation}
\end{definition}

In particular, any Hessenberg variety $\Hess(A,h)$ is, by definition,
an algebraic subset of the flag variety $\Flags(\CC^n)$.  It is
straightforward to see that $\Hess(A,h)$ and $\Hess(gAg^{-1},h)$ are
isomorphic $\forall g\in \GL_n(\CC)$, so we frequently assume
without loss of generality that $A$ is in standard Jordan canonical
form with respect to the standard basis on $\CC^n$. 

Since $H(h)\subseteq \mathfrak{gl}_n(\CC)$ is stable under the action of $B$, the quotient space $\mathfrak{gl}_n(\CC)/H(h)$ admits a $B$-action by $b\cdot \overline{X}=\overline{\text{Ad}(b)X}$ for $b\in B$ and $X\in \mathfrak{gl}_n(\CC)$, where $\overline{X}$ denotes the image of $X$ in $\mathfrak{gl}_n(\CC)/H(h)$.
So we have the following vector bundle over $\GL_n(\CC)/B$:
\begin{equation*}
\GL_n(\CC)\times_B (\mathfrak{gl}_n(\CC)/H(h))
\end{equation*}
where $B$ acts on the product $\GL_n(\CC)\times (\mathfrak{gl}_n(\CC)/H(h))$ by $(M,\overline{X})\cdot b = (Mb,\overline{b^{-1}Xb})$ for $b\in B$ and $(M,\overline{X})\in\GL_n(\CC)\times (\mathfrak{gl}_n(\CC)/H(h))$.
A matrix $A\in\mathfrak{gl}_n(\CC)$ defines a section of this vector bundle by
\begin{equation}\label{eq:section}
s_A \colon \GL_n(\CC)/B \rightarrow \GL_n(\CC)\times_B (\mathfrak{gl}_n(\CC)/H(h))
\quad ; \quad 
MB \rightarrow [M,\overline{M^{-1}AM}].
\end{equation}
Now it is clear from \eqref{eq:def Hess conj} that 
\begin{equation}\label{eq:def Hess zero}
\Hess(A,h) = \{ MB \in \GL_n(\CC)/B \mid s_A(MB)=0 \}.
\end{equation}
Namely, $\Hess(A,h)$ is the zero set $Z(s_A)$ of the section $s_A$.

In this manuscript we discuss two important special cases of
Hessenberg varieties: the regular nilpotent Hessenberg varieties and
the regular semisimple Hessenberg varieties.

\begin{definition}\label{def:reg nilp}
  A Hessenberg variety $\Hess(A,h)$ is called \textbf{regular
    nilpotent} if $A$ is a principal nilpotent operator. Equivalently,
  the Jordan canonical form of $A$ has a single Jordan block with
  eigenvalue zero, i.e., up to a change of basis $A$ is of the form:
\begin{equation*}
\begin{pmatrix}
0 & 1 & &  & \\
 & 0 &  1  & \\
 &    & \ddots & \ddots & \\
 &    &   & 0 & 1 \\
 &    &  & & 0 \\ 
\end{pmatrix}
\end{equation*}
For the remainder of this paper we let $N$ denote the matrix (operator)
above. 
\end{definition}

Regular nilpotent Hessenberg varieties are known to be irreducible
\cite[Lemma 7.1]{AT}, and they are the subject of
Section~\ref{sec:ideal} of this
paper. When we study families of Hessenberg varieties in
Section~\ref{sec:family}, the following type will also become
relevant.

\begin{definition}\label{regularssdefn}
  A Hessenberg variety $\Hess(A,h)$ is called \textbf{regular
    semisimple} if $A$ is a semisimple operator with distinct
  eigenvalues.  Equivalently, there is a basis of $\CC^n$ with respect
  to which $A$ is diagonal with pairwise distinct entries along the diagonal.
\end{definition}

We will need the following terminology from 
  \cite[Definition 4.4]{Drellich}. 
\begin{definition}
  Let $h \colon [n]\to [n]$ be a Hessenberg function. If
  $h(j) \geq j+1$ for $j\in\{1,2,\ldots,n-1\}$, then we say that $h$
  is \textbf{indecomposable}.
\end{definition}

Finally, we give the definition of a special case of a regular nilpotent
Hessenberg variety which is studied in more detail in
Section~\ref{sec:NOBY Peterson}.

\begin{definition}\label{def:Peterson}
  When $h$ is of the form $h(j)=j+1$ for $j\in\{1,2,\ldots,n-1\}$, the
  corresponding regular nilpotent Hessenberg variety is called a
  \textbf{Peterson variety}.
\end{definition}

\section{The defining ideals of regular nilpotent Hessenberg varieties}\label{sec:ideal}

In this section, we will show that the zero scheme $Z(s_A)$ of the section $s_A$ of the vector bundle $\GL_n(\CC)\times_B (\mathfrak{gl}_n(\CC)/H(h))$ described in \eqref{eq:section} is reduced as a scheme. As a corollary, we will provide 
explicit lists of (local) generators for the defining ideals of 
$\Hess(N,h)$, considered as subvarieties of $\Flags(\CC^n)$.
Since we already know that $\Hess(N,h)=Z(s_A)$ in $\Flags(\CC^n)$ as in~\eqref{eq:def Hess zero}, a local trivialization of the vector bundle above produces an explicit list of polynomials which cut out $\Hess(N,h)$ set-theoretically; the issue which we must address is whether the
ideal that these polynomials generate is radical, or, whether the
relevant quotient ring is reduced. The main content of this section,
recorded in Proposition~\ref{prop:main claim for Z} and Proposition~\ref{prop:defining ideal of Hessenberg}, is to show that in fact
the quotient rings associated to our lists of polynomials are
reduced and thus we have found generators for the defining ideals of
our varieties. We can then easily conclude that $\Hess(N,h)$ is a local complete intersection (Corollary~\ref{cor:LCI}). 

Recall from Definition \ref{def:reg nilp} that $N$ is the $n\times n$ regular nilpotent matrix in Jordan canonical form.
We define the following. 
\begin{definition}
Let $\mathcal{Z}(N,h)$ denote the \emph{zero scheme} in $G/B$ of the section $s_N$.
\end{definition}
Locally, the section $s_N$ is a collection of (local) regular functions, and the scheme $\mathcal{Z}(N,h)$ is locally the zero scheme of those functions (cf. also \cite[Appendix B.3.2]{Fulton IntTh}). Note that, a priori, $\mathcal{Z}(N,h)$ may be nonreduced. We now describe an explicit local presentation of $\mathcal{Z}(N,h)$. 

It is
well-known that $\Flags(\CC^n)$ can be covered by affine coordinate
patches, each isomorphic to $\A^{n(n-1)/2}$. Let 
\begin{align*}
  U^-  &:= \left\{ M =
      \left.
      \begin{pmatrix}
        1 &  &  &  &  \\
        \star & 1 &  &  &  \\
        \vdots & \vdots & \ddots & & \\
        \star & \star & \dots & 1 & \\
        \star & \star & \cdots & \star & 1
      \end{pmatrix}
      \
      \right|
      \begin{matrix} \textup{$M$ is lower-triangular} \\ \textup{ with $1$'s along the diagonal} \end{matrix} 
\right\} \\
&\cong \A^{n(n-1)/2} \subseteq \operatorname{Mat}(n\times n, \CC).
\end{align*}
Then the map $U^-  \to \Flags(\CC^n) \cong \GL_n(\CC)/B$ given by
$M \in U^-  \mapsto MB \in \GL_n(\CC)/B$, is an open embedding. By slight
abuse of notation we denote also by $U^- $ its image in $\Flags(\CC^n)$.
The set of translates
$\{\mathcal{N}_w := w U^- \}$ of $U^- $ by the permutations
$w \in \mathfrak{S}_n$, along with the embeddings
$\Psi_w\colon U^-  \cong \A^{n(n-1)/2} \xrightarrow{\cong} \mathcal{N}_w \subseteq
\Flags(\CC^n)$ sending $M \mapsto w MB$, form an open cover of
$\Flags(\CC^n)$ :
\begin{align*}
\Flags(\CC^n) = \bigcup_{w\in\mathfrak{S}_n} \mathcal{N}_w.
\end{align*}

Using the bijection
$U^-  \xrightarrow{\cong} \mathcal{N}_w$, a point in
$\mathcal{N}_w$ is uniquely identified with the $w$-translate of a
lower-triangular matrix with 1's along the diagonal.  Therefore a
point in $\mathcal{N}_w$ is uniquely determined by a matrix
$(x_{i,j})$ whose entries are subject to the following relations
    \begin{equation}
      \label{eq:11}
      \begin{split}
        &x_{w(j),j} = 1, \quad \forall j\in [n],\\
        &x_{w(i),j} = 0, \quad \forall i,j\in [n] : j>i.
      \end{split}
    \end{equation}
    Thus the coordinate ring of $\mathcal{N}_w$, denoted by
    $\CC [ \mathbf{x}_w]$, is isomorphic to the quotient
    of the polynomial ring $\CC [x_{i,j}]$ by the relations
    \eqref{eq:11}. Observe that $\CC [ \mathbf{x}_w]$ is isomorphic to
    a polynomial ring in the $n(n-1)/2$ variables $x_{i,j}$ not
    covered by the relations \eqref{eq:11}.   

\begin{example}\label{ex:matrix notation}
Let $n=4$ and $w=(2,4,1,3)\in\mathfrak{S}_4$ in the standard one-line notation.
An element $M$ of $\mathcal{N}_w=wU^- $ can be written as
\begin{equation*}
    wM=
    \begin{pmatrix}
      x_{1,1} & x_{1,2} & 1 & 0 \\
      1 & 0 & 0 & 0\\
      x_{3,1} & x_{3,2} & x_{3,3} & 1 \\    
      x_{4,1} & 1 & 0 & 0
    \end{pmatrix}.
  \end{equation*}
\end{example}

To describe the local presentation of the zero scheme $\mathcal{Z}(N,h)$ in the neighborhood $\mathcal{N}_w$, we make the following definition. 

\begin{definition}\label{def:ideal}
  Let $w \in \mathfrak{S}_n$ and let $i, j \in [n]$ with $i>h(j)$. We define the
  polynomial $f^w_{i,j} \in \CC[\mathbf{x}_w]$ by 
\[
f^w_{i,j} := \left( (wM)^{-1} N (wM) \right)_{i,j}
\]
where here the $(k,\ell)$-th matrix entries of $M$ for $k > \ell$ are
viewed as variables. We also define the ideal 
\[
J_{w,h} := \langle f^w_{i,j} \mid i > h(j) \rangle \subseteq \CC[\mathbf{x}_w]
\]
to be the ideal in $\CC[\mathbf{x}_w]$ generated by the $f^w_{i,j}$. 
\end{definition}

From the explicit description of $s_N$ given in~\eqref{eq:section}, the following Lemma is straightforward.

\begin{lemma}\label{lem:local description}
Let $w \in \mathfrak{S}_n$ and $\mathcal{N}_w = \Spec \CC[\mathbf{x}_w]$ as above. 
Then $\mathcal{Z}(N,h)\cap \Spec \CC[\mathbf{x}_w] \cong \Spec \CC[\mathbf{x}_w]/J_{w,h}$. In particular, we obtain an open affine cover
\[
\mathcal{Z}(N,h) = \bigcup_{w\in\Sn} \Spec \CC[\mathbf{x}_w]/J_{w,h}.
\]
\end{lemma}

\begin{example}\label{ex:explicit generator}
Let $n=4$ and $w=(2,4,1,3)\in\mathfrak{S}_4$, continued from Example~\ref{ex:matrix notation}.
Then it is easy to check that 
\begin{align*}
    (wM)^{-1} &= 
    \begin{pmatrix}
      x_{1,1} & x_{1,2} & 1 & 0 \\
      1 & 0 & 0 & 0\\
      x_{3,1} & x_{3,2} & x_{3,3} & 1 \\    
      x_{4,1} & 1 & 0 & 0
    \end{pmatrix}^{-1} \\
    &=
     \begin{pmatrix}
      0 & 1 & 0 & 0 \\
      0 & -x_{4,1} & 0 & 1\\
      1 & -x_{1,1}+x_{1,2}x_{4,1} & 0 & -x_{1,2} \\    
      -x_{3,3} & y & 1 & -x_{3,2}+x_{1,2}x_{3,3}
    \end{pmatrix}
\end{align*}
where $y=-x_{3,1}+x_{1,1}x_{3,3}+x_{4,1}(x_{3,2}-x_{1,2}x_{3,3})$. 
So, for example, we have
\begin{equation}\label{eq:example of generator}
\begin{split}
    f^w_{4,1} &= 
    \left( (wM)^{-1} N (wM) \right)_{4,1}\\
    &= -x_{3,3} + x_{3,1}(-x_{3,1}+x_{1,1}x_{3,3}+x_{4,1}(x_{3,2}-x_{1,2}x_{3,3})) + x_{4,1},\\
    f^w_{4,2} &= 
    \left( (wM)^{-1} N (wM) \right)_{4,2} \\
    &= x_{3,2}(-x_{3,1}+x_{1,1}x_{3,3}+x_{4,1}(x_{3,2}-x_{1,2}x_{3,3})) + 1.
\end{split}
\end{equation}
So if $h=(3,3,4,4)$, then we have $J_{w,h}=\langle f^w_{4,1}, f^w_{4,2} \rangle$ with these polynomials.
\end{example}

We now state the main result of this section. 
\begin{proposition}\label{prop:main claim for Z}

Let $h\colon [n]\rightarrow[n]$ be an indecomposable Hessenberg function. Then the zero scheme $\mathcal{Z}(N,h)$ of the section $s_N$ described at \eqref{eq:section} is reduced. 

\end{proposition}
Combining this with the local description of $\mathcal{Z}(N,h)$ given at Lemma \ref{lem:local description}, we obtain the following. 
  \begin{proposition}\label{prop:defining ideal of Hessenberg}
   Let $h\colon [n]\rightarrow[n]$ be an indecomposable Hessenberg function. For every $w\in\Sn$, 
   the ring $\CC[\mathbf{x}_w]/J_{w,h}$ is the coordinate ring of the
   subvariety $\mathcal{N}_{w,h} := \Hess(N,h) \cap \mathcal{N}_w$ of
   $\mathcal{N}_w$. In particular, the ideal
     $J_{w,h}$ is radical and is the defining ideal of
     $\mathcal{N}_{w,h}$. 
  \end{proposition}

\begin{remark}
Using the language and techniques of degeneracy loci, 
it is shown in \cite{AT} that $\mathcal{Z}(N,h)$ is reduced, 
for the special case when the Hessenberg function is of the form
$h=(k,n,\dots,n)$ for some $2 \leq k \leq n$ \cite[Theorem 7.6]{AT}.
\end{remark}

The necessity of the indecomposability hypothesis can be seen from a
small example. 

\begin{example}
  Let $n=2$ and $h=(1,2)$. We have
  $J_{\id,h} = \langle f^{\id}_{2,1}\rangle \subseteq \CC [x_{2,1}]$ where
  $f^{\id}_{2,1} = -x_{2,1}^2$. Clearly the ring $\CC [x_{2,1}] / J_{\id,h}$ is
  not reduced, so it is not the coordinate ring of
  $\mathcal{N}_{\id,h}$.
\end{example}

We now prove Proposition \ref{prop:main claim for Z}. For this, we recall the following property of the regular nilpotent Hessenberg variety $\Hess (N,h)$.
\begin{proposition}\label{prop:}
$($\cite[Lemma 7.1]{AT}$)$
Let $h\colon [n] \to [n]$ be a Hessenberg function.
Then $\Hess (N,h)$ is irreducible of
  dimension $\sum_{i=1}^n (h(i)-i)$.
\end{proposition}
This proposition implies that the zero scheme $\mathcal{Z}(N,h)$ is irreducible as well, and has the expected codimension. Hence the following is immediate from \cite[\S 18.5 and Proposition
  18.13]{Eisenbud} (cf. \cite[Theorem 8.2]{Fulton Schubert}). 
\begin{lemma}\label{lem:CM}
The scheme $\mathcal{Z}(N,h)$ is a local complete intersection and hence Cohen-Macaulay.
\end{lemma}
Thus, to prove the reducedness of $\mathcal{Z}(N,h)$, it is enough to prove that $\mathcal{Z}(N,h)$ is generically reduced (\cite[Exercise 18.9]{Eisenbud}). Since $\mathcal{Z}(N,h)$ is irreducible, we only need to find a single reduced point of the scheme $\mathcal{Z}(N,h)$. To do this, we focus our attention on the neighborhood $\mathcal{N}_{w_0} \cong \Spec \CC [\mathbf{x}_{w_0}]$ where $w_0$ is the longest element in $\mathfrak{S}_n$. The following is the most
important computation in our argument.
  
\begin{lemma}
  \label{pro:2}
  Let $h\colon [n] \to [n]$ be an indecomposable Hessenberg function.
  Then the ring $\CC [\mathbf{x}_{w_0}] / J_{w_0,h}$ is isomorphic to
  a polynomial ring, hence it is reduced.
\end{lemma}

  It is already known that the 
intersection $\Hess (N,h)\cap \mathcal{N}_{w_0}$ of the
variety $\Hess(N,h)$ with the affine coordinate patch around
$w_0$ is isomorphic as a variety to a complex affine space
(\cite{ty} and \cite{precup13a}). The point of Lemma~\ref{pro:2}
is that $J_{w_0,h}$ is its defining ideal, and that its generators take a
particular form. Before proving Lemma~\ref{pro:2}, we give some concrete examples. 

\begin{example}
  Let $n=4$ and $h=(3,3,4,4)$. The longest element of $\mathfrak{S}_4$ is the permutation $w_0=(4,3,2,1)$ in one-line notation. The coordinate ring of
  $\mathcal{N}_{w_0}$ is
  $$\CC [\mathbf{x}_{w_0}] \cong \CC
  [x_{1,1},x_{1,2},x_{1,3},x_{2,1},x_{2,2},x_{3,1}],$$ and a point in
  $\mathcal{N}_{w_0}$ is determined by a matrix
  \begin{equation*}
    M=
    \begin{pmatrix}
      x_{1,1} & x_{1,2} & x_{1,3} & 1\\
      x_{2,1} & x_{2,2} & 1 & 0\\
      x_{3,1} & 1 & 0 & 0\\
      1 & 0 & 0 & 0
    \end{pmatrix}.
  \end{equation*}
  Given the form of $M$, it is easy to see that its inverse must have
  the form
  \begin{align}\label{eq:yij}
    M^{-1} =
    \begin{pmatrix}
      0 & 0 & 0 & 1\\
      0 & 0 & 1 & y_{3,1}\\
      0 & 1 & y_{2,2} & y_{2,1}\\
      1 & y_{1,3} & y_{1,2} & y_{1,1}
    \end{pmatrix}.
  \end{align}
  Starting from the matrix equality $M^{-1} M=(\delta_{i,j})$, and
  comparing entries we can obtain expressions for the $y_{i,j}$ in
  terms of the $x_{i,j}$. For example,
  \begin{equation*}
    \begin{split}
      &y_{1,3} = -x_{1,3},\\
      &y_{1,2} = -x_{1,2} - y_{1,3} x_{2,2} = -x_{1,2} + x_{1,3}
      x_{2,2}.
    \end{split}
  \end{equation*}
  It is also straightforward to see that each $y_{i,j}$ depends only on the variables $x_{k,\ell}$
  with $k\geq i$ and $\ell\geq j$. 
  Graphically, this says that $y_{i,j}$
  depends only on $x_{i,j}$ and variables located to the right or
  below $x_{i,j}$ in the matrix $M$; for example, $y_{1,2}$ depends
  only on the variables contained in the bounded region depicted in
  Figure \ref{fig:1}.
  \begin{figure}[h]
    \begin{center}
      \begin{tikzpicture}
        \begin{scope}[xscale=0.7,yscale=0.6]
          \draw (0.5,3.5) node {$x_{1,1}$};
          \draw (1.5,3.5) node {$x_{1,2}$};
          \draw (2.5,3.5) node {$x_{1,3}$};
          \draw (3.5,3.5) node {$1$};
          \draw (0.5,2.5) node {$x_{2,1}$};
          \draw (1.5,2.5) node {$x_{2,2}$};
          \draw (2.5,2.5) node {$1$};
          \draw (3.5,2.5) node {$0$};
          \draw (0.5,1.5) node {$x_{3,1}$};
          \draw (1.5,1.5) node {$1$};
          \draw (2.5,1.5) node {$0$};
          \draw (3.5,1.5) node {$0$};
          \draw (0.5,0.5) node {$1$};
          \draw (1.5,0.5) node {$0$};
          \draw (2.5,0.5) node {$0$};
          \draw (3.5,0.5) node {$0$};
          \draw[thick] (1,2)--(2,2)--(2,3)--(3,3)--(3,4)--(1,4)--(1,2);
        \end{scope}
      \end{tikzpicture}
    \end{center}

    \caption{Variables appearing in the expression of $y_{1,2}$}
    \label{fig:1}
  \end{figure}

  Now we describe the generators of
  $J_{w_0,h} = \langle f^{w_0}_{4,1}, f^{w_0}_{4,2}\rangle$. We have
  \begin{equation*}
    M^{-1} N M =
    \begin{pmatrix}
      0 & 0 & 0 & 1\\
      0 & 0 & 1 & y_{3,1}\\
      0 & 1 & y_{2,2} & y_{2,1}\\
      1 & y_{1,3} & y_{1,2} & y_{1,1}
    \end{pmatrix}
    \begin{pmatrix}
      x_{2,1} & x_{2,2} & 1 & 0\\
      x_{3,1} & 1 & 0 & 0\\
      1 & 0 & 0 & 0\\
      0 & 0 & 0 & 0
    \end{pmatrix}
  \end{equation*}
and from this we get
  \begin{equation*}
    \begin{split}
      &f^{w_0}_{4,1} = (M^{-1} N M)_{4,1} = x_{2,1} + y_{1,3}
      x_{3,1} +y_{1,2} =
      x_{2,1} - x_{1,3} x_{3,1} -x_{1,2} + x_{1,3} x_{2,2},\\
      &f^{w_0}_{4,2} = (M^{-1} N M)_{4,2} = x_{2,2} + y_{1,3}
      = x_{2,2} - x_{1,3}.
    \end{split}
  \end{equation*}
  We deduce that $x_{2,1}$ and $x_{2,2}$ are determined by the other
  variables and conclude that 
  $\CC [\mathbf{x}_{w_0}] / J_{w_0,h} \cong \CC
  [x_{1,1},x_{1,2},x_{1,3},x_{3,1}]$ is a polynomial ring and in
  particular is reduced.
It is possible to easily visualize, using the Hessenberg diagram, the
variables which turn out to be dependent on other variables and hence
``vanish'' in the quotient $\CC[\mathbf{x}_{w_0}] / J_{w_0, h}$, as
illustrated in Figure~\ref{fig:2} for this example. Specifically, we
can first cross out any box which is \emph{not}
contained in the Hessenberg diagram for $h=(3344)$; see the left
diagram in Figure~\ref{fig:2}. We then flip the
picture upside down (so that, in this case, the boxes in positions
$(1,1)$ and $(1,2)$ are now crossed out), and finally
shift the entire picture downwards by one row. In this case we end up
with a picture, as in the right-hand diagram in Figure~\ref{fig:2},
with the boxes in positions $(2,1)$ and $(2,2)$ crossed out. Then the
variables corresponding to the crossed-out boxes are the ones which
vanish in the quotient, and in fact (by the computation above) they
are dependent on the (non-crossed-out) variables appearing either
below it within the same column, or in a column to its right in a row
at most one above it.

  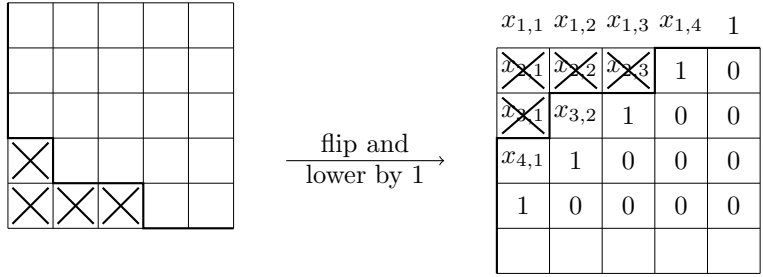
\begin{figure}[h]
    \begin{center}
      \begin{tikzpicture}
        \begin{scope}[xscale=0.6,yscale=0.6]
          \draw[ultra thin] (0,0)--(4,0);
          \draw[ultra thin] (0,1)--(4,1);
          \draw[ultra thin] (0,2)--(4,2);
          \draw[ultra thin] (0,3)--(4,3);
          \draw[ultra thin] (0,4)--(4,4);
          \draw[ultra thin] (0,0)--(0,4);
          \draw[ultra thin] (1,0)--(1,4);
          \draw[ultra thin] (2,0)--(2,4);
          \draw[ultra thin] (3,0)--(3,4);
          \draw[ultra thin] (4,0)--(4,4);
          \draw[thick] (0,4)--(0,1)--(2,1)--(2,0)--(4,0);
          \draw[thick] (0.1,0.1)--(0.9,0.9);
          \draw[thick] (0.1,0.9)--(0.9,0.1);
          \draw[thick] (1.1,0.1)--(1.9,0.9);
          \draw[thick] (1.1,0.9)--(1.9,0.1);
        \end{scope}
        \begin{scope}[shift={(3.2,0)},xscale=0.7,yscale=0.6]
          \draw[->] (0,1.5) to node[midway,align=center]
          {flip and\\lower by 1} (3,1.5);
        \end{scope}
        \begin{scope}[shift={(6,0)},xscale=0.7,yscale=0.6]
          \draw (0.5,3.5) node {$x_{1,1}$};
          \draw (1.5,3.5) node {$x_{1,2}$};
          \draw (2.5,3.5) node {$x_{1,3}$};
          \draw (3.5,3.5) node {$1$};
          \draw (0.5,2.5) node {$x_{2,1}$};
          \draw (1.5,2.5) node {$x_{2,2}$};
          \draw (2.5,2.5) node {$1$};
          \draw (3.5,2.5) node {$0$};
          \draw (0.5,1.5) node {$x_{3,1}$};
          \draw (1.5,1.5) node {$1$};
          \draw (2.5,1.5) node {$0$};
          \draw (3.5,1.5) node {$0$};
          \draw (0.5,0.5) node {$1$};
          \draw (1.5,0.5) node {$0$};
          \draw (2.5,0.5) node {$0$};
          \draw (3.5,0.5) node {$0$};
          \draw[ultra thin] (0,-1)--(4,-1);
          \draw[ultra thin] (0,0)--(4,0);
          \draw[ultra thin] (0,1)--(4,1);
          \draw[ultra thin] (0,2)--(4,2);
          \draw[ultra thin] (0,3)--(4,3);
          \draw[ultra thin] (0,-1)--(0,3);
          \draw[ultra thin] (1,-1)--(1,3);
          \draw[ultra thin] (2,-1)--(2,3);
          \draw[ultra thin] (3,-1)--(3,3);
          \draw[ultra thin] (4,-1)--(4,3);
          \draw[thick] (0,-1)--(0,2)--(2,2)--(2,3)--(4,3);
          \draw[thick] (0.1,2.1)--(0.9,2.9);
          \draw[thick] (0.1,2.9)--(0.9,2.1);
          \draw[thick] (1.1,2.1)--(1.9,2.9);
          \draw[thick] (1.1,2.9)--(1.9,2.1);
        \end{scope}
      \end{tikzpicture}
    \end{center}
  
    \caption{Variables killed in $\CC[\mathbf{x}_{w_0}] / J_{w_0,h}$}
    \label{fig:2}
  \end{figure}
\end{example}

\begin{example}
  \label{exa:1}
  Let $n=5$ and $h=(3,4,4,5,5)$. The diagram in Figure \ref{fig:3}
  predicts that
  $\CC [\mathbf{x}_{w_0}] / J_{w_0,h} \cong \CC
  [x_{1,1},x_{1,2},x_{1,3},x_{1,4},x_{3,2},x_{4,1}]$. Indeed the
  generators of $J_{w_0,h}$ are
  \begin{equation*}
    \begin{split}
      f^{w_0}_{5,1} &={x}_{2,1}-{x}_{1,2}-{x}_{1,3}
      {x}_{4,1}+{x}_{1,3} {x}_{3,2}
      -{x}_{1,4} {x}_{3,1}\\
      &\quad +{x}_{1,4} {x}_{2,2}+{x}_{1,4} {x}_{2,3} {x}_{4,1},
      -{x}_{1,4} {x}_{2,3} {x}_{3,2}\\
      f^{w_0}_{5,2}&={x}_{2,2}-{x}_{1,3}-{x}_{1,4} {x}_{3,2}+{x}_{1,4} {x}_{2,3}\\
      f^{w_0}_{5,3}&={x}_{2,3}-{x}_{1,4}\\
      f^{w_0}_{4,1}&={x}_{3,1}-{x}_{2,2}-{x}_{2,3} {x}_{4,1}+{x}_{2,3}
      {x}_{3,2}.
    \end{split}
  \end{equation*}
  Again, we see that $\CC [\mathbf{x}_{w_0}] / J_{w_0,h}$ is reduced.
 Following the method outlined in the previous example, we see that
 the variables which vanish in the quotient are $x_{2,1}, x_{2,2},
 x_{2,3}$ and $x_{3,1}$. See Figure~\ref{fig:3}. 

  \begin{figure}[h]
    \begin{center}
      \begin{tikzpicture}
        \begin{scope}[xscale=0.6,yscale=0.6]
          \draw[ultra thin] (0,0)--(5,0);
          \draw[ultra thin] (0,1)--(5,1);
          \draw[ultra thin] (0,2)--(5,2);
          \draw[ultra thin] (0,3)--(5,3);
          \draw[ultra thin] (0,4)--(5,4);
          \draw[ultra thin] (0,5)--(5,5);
          \draw[ultra thin] (0,0)--(0,5);
          \draw[ultra thin] (1,0)--(1,5);
          \draw[ultra thin] (2,0)--(2,5);
          \draw[ultra thin] (3,0)--(3,5);
          \draw[ultra thin] (4,0)--(4,5);
          \draw[ultra thin] (5,0)--(5,5);
          \draw[thick] (0,5)--(0,2)--(1,2)--(1,1)--(3,1)--(3,0)--(5,0);
          \draw[thick] (0.1,0.1)--(0.9,0.9);
          \draw[thick] (0.1,0.9)--(0.9,0.1);
          \draw[thick] (1.1,0.1)--(1.9,0.9);
          \draw[thick] (1.1,0.9)--(1.9,0.1);
          \draw[thick] (2.1,0.1)--(2.9,0.9);
          \draw[thick] (2.1,0.9)--(2.9,0.1);
          \draw[thick] (0.1,1.1)--(0.9,1.9);
          \draw[thick] (0.1,1.9)--(0.9,1.1);
        \end{scope}
        \begin{scope}[shift={(3.7,0)},xscale=0.7,yscale=0.6]
          \draw[->] (0,1.5) to node[midway,align=center]
          {flip and\\lower by 1} (3,1.5);
        \end{scope}
        \begin{scope}[shift={(6.5,0)},xscale=0.7,yscale=0.6]
          \draw (0.5,4.5) node {$x_{1,1}$};
          \draw (1.5,4.5) node {$x_{1,2}$};
          \draw (2.5,4.5) node {$x_{1,3}$};
          \draw (3.5,4.5) node {$x_{1,4}$};
          \draw (4.5,4.5) node {$1$};
          \draw (0.5,3.5) node {$x_{2,1}$};
          \draw (1.5,3.5) node {$x_{2,2}$};
          \draw (2.5,3.5) node {$x_{2,3}$};
          \draw (3.5,3.5) node {$1$};
          \draw (4.5,3.5) node {$0$};
          \draw (0.5,2.5) node {$x_{3,1}$};
          \draw (1.5,2.5) node {$x_{3,2}$};
          \draw (2.5,2.5) node {$1$};
          \draw (3.5,2.5) node {$0$};
          \draw (4.5,2.5) node {$0$};
          \draw (0.5,1.5) node {$x_{4,1}$};
          \draw (1.5,1.5) node {$1$};
          \draw (2.5,1.5) node {$0$};
          \draw (3.5,1.5) node {$0$};
          \draw (4.5,1.5) node {$0$};
          \draw (0.5,0.5) node {$1$};
          \draw (1.5,0.5) node {$0$};
          \draw (2.5,0.5) node {$0$};
          \draw (3.5,0.5) node {$0$};
          \draw (4.5,0.5) node {$0$};
          \draw[ultra thin] (0,-1)--(5,-1);
          \draw[ultra thin] (0,0)--(5,0);
          \draw[ultra thin] (0,1)--(5,1);
          \draw[ultra thin] (0,2)--(5,2);
          \draw[ultra thin] (0,3)--(5,3);
          \draw[ultra thin] (0,4)--(5,4);
          \draw[ultra thin] (0,-1)--(0,4);
          \draw[ultra thin] (1,-1)--(1,4);
          \draw[ultra thin] (2,-1)--(2,4);
          \draw[ultra thin] (3,-1)--(3,4);
          \draw[ultra thin] (4,-1)--(4,4);
          \draw[ultra thin] (5,-1)--(5,4);
          \draw[thick] (0,-1)--(0,2)--(1,2)--(1,3)--(3,3)--(3,4)--(5,4);
          \draw[thick] (0.1,3.1)--(0.9,3.9);
          \draw[thick] (0.1,3.9)--(0.9,3.1);
          \draw[thick] (1.1,3.1)--(1.9,3.9);
          \draw[thick] (1.1,3.9)--(1.9,3.1);
          \draw[thick] (2.1,3.1)--(2.9,3.9);
          \draw[thick] (2.1,3.9)--(2.9,3.1);
          \draw[thick] (0.1,2.1)--(0.9,2.9);
          \draw[thick] (0.1,2.9)--(0.9,2.1);
        \end{scope}
      \end{tikzpicture}
    \end{center}
  
    \caption{Variables killed in $\CC[\mathbf{x}_{w_0}] / J_{w_0,h}$}
    \label{fig:3}
  \end{figure}
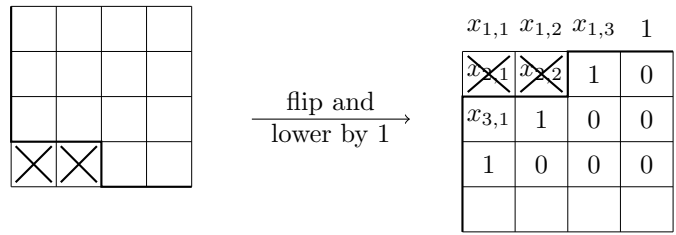
\end{example}

\begin{proof}[Proof of Lemma~\ref{pro:2}]
  Let $M=(x_{i,j})$ determine a point in $\mathcal{N}_{w_0,h}$. Recall
  that, as elements of $\CC[\mathbf{x}_{w_0}]$, the variables
  $x_{i,j}$ are subject to the following relations:
  \begin{itemize}
  \item $x_{i,n+1-i} = 1, \quad \forall i\in [n]$;
  \item $x_{i,j} = 0, \quad\forall i,j\in [n] : i>n+1-j$.
  \end{itemize}
  For all $i,j\in [n]$, we have
  $(M^{-1} M)_{n+1-i,j} = \delta_{n+1-i,j}$. This equality can be
  written more explicitly as
  \begin{equation}
    \label{eq:2}
    y_{i,j} + \sum_{k=1}^{n-j} y_{i,n+1-k} x_{k,j} = \delta_{n+1-i,j},
  \end{equation}
  where $y_{i,j} := (M^{-1})_{n+1-i,n+1-j}$ (see \eqref{eq:yij} or \eqref{eq:general yij} below for visualizations of this indexing).

  For all $i,j \in [n]$, the polynomials $y_{i,j}$ have the following
  properties:
  \begin{enumerate}[label=(\roman*)]
  \item\label{item:1} $y_{i,n+1-i} = 1$;
  \item\label{item:2} $y_{i,j} = 0$, whenever $i>n+1-j$;
  \item\label{item:3} $y_{i,j}$ is a polynomial in the variables
    $x_{k,l}$ with $k\geq i$ and $l \geq j$.
  \end{enumerate}
  These properties follow from equation \eqref{eq:2} using an
  elementary inductive argument.
  Using properties \ref{item:1} and \ref{item:2}, we deduce that
  \begin{align}\label{eq:general yij}
    M^{-1} = 
    \begin{pmatrix}
      & & & & 1\\
      & & & 1 & y_{n-1,1}\\
      & & \iddots & \vdots & \vdots\\
      & 1 & \ldots & y_{2,2} & y_{2,1}\\
      1 & y_{1,n-1} & \ldots & y_{1,2} & y_{1,1}
    \end{pmatrix}.
  \end{align}

  Let us compute the polynomial $f^{w_0}_{n+1-i,j}$. We have
  \begin{equation*}
    \begin{split}
      N M &=
      \begin{pmatrix}
        0 & 1 & & & \\
        & 0 & 1 & \\
        & & \ddots & \ddots & \\
        & & & 0 & 1\\
        & & & & 0
      \end{pmatrix}
      \begin{pmatrix}
        x_{1,1} & x_{1,2} & \ldots & x_{1,n-1} & 1\\
        x_{2,1} & x_{2,2} & \ldots & 1 & \\
        \vdots & \vdots & \iddots & & \\
        x_{n-1,1} & 1 & & & \\
        1 &
      \end{pmatrix}\\
      &=
      \begin{pmatrix}
        x_{2,1} & x_{2,2} & \ldots & 1 & 0\\
        \vdots & \vdots & \iddots & \iddots & \\
        x_{n-1,1} & 1 & \iddots & & \\
        1 & 0 & \\
        0
      \end{pmatrix}.
    \end{split}
  \end{equation*}
  The ideal $J_{w_0,h}$ is generated by the polynomials
  $f^{w_0}_{n+1-i,j}$ having $n+1-i > h(j)$.
  With this choice of indices, we obtain
  \begin{equation*}
    f^{w_0}_{n+1-i,j} = (M^{-1} N M)_{n+1-i,j} 
    =\sum_{k=i}^{n-j} x_{k+1,j} y_{i,n+1-k}.
  \end{equation*}
  Since we are dealing with $f^{w_0}_{n+1-i,j}$ for $n+1-i>h(j)$, we have $n+1-i>j$ by combining with $h(j)\geq j$.
  Therefore a
  generator of $J_{w_0,h}$ has the form
  \begin{equation}
    \label{eq:5}
    f^{w_0}_{n+1-i,j} =x_{i+1,j} + \sum_{k=i+1}^{n-j} x_{k+1,j} y_{i,n+1-k}.
  \end{equation}
  Namely, the first summand $x_{i+1,j}$ always appears with $y_{i,n+1-i}=1$.
  Now, since $h$ is indecomposable, we have $h(j) \geq j+1$. In fact we have $i<n-j$ from the same reasoning as above, so that $x_{i+1,j}$ is a coordinate function on $\mathcal{N}_{w_0}$ (cf.\ \eqref{eq:11}).
  The variable $x_{k+1,j}$ appearing in the summation has row index
  $k+1 \geq i+2$. As for $y_{i,n+1-k}$, it depends only on variables
  $x_{p,q}$ with row index $p\geq i$ and column index $q\geq j+1$.
  This follows from property \ref{item:3} combined with the
  observation that $k\leq n-j$ implies $n+1-k\geq j+1$.  We conclude
  that the summation appearing in equation \eqref{eq:5} depends only
  on variables $x_{p,q}$ with $q\geq j+1$ and $p \geq i$, or $q=j$ and $p\geq i+2$.
 
Finally, the above discussion and a simple inductive argument imply that setting $f^{w_0}_{n+1-i,j}$
equal to $0$ has the effect of eliminating the variables $x_{i+1,j}$
  from the quotient $\CC [\mathbf{x}_{w_0}] / J_{w_0,h}$ and there are
  no further relations on the remaining variables. 
  Namely, $\CC [\mathbf{x}_{w_0}] / J_{w_0,h}$ is isomorphic to the
  polynomial ring
  \begin{equation*}
    \CC [x_{i,j} \mid 1\leq i,j\leq n-1,
    \ i\notin \{2,3,\ldots, n+1-h(j)\}],
  \end{equation*}
which in particular is reduced, as was to be shown. It also follows
that $J_{w_0,h}$ is radical and is the defining ideal of
$\mathcal{N}_{w_0,h}$ in $\mathcal{N}_{w_0}$.
\end{proof}

Motivated by the proof of Lemma~\ref{pro:2}, we introduce
the following terminology which will be useful in
Section~\ref{sec:flags}: the set $\{x_{i,j} \mid 1 \leq i , j \leq n-1, i \in
\{2,3,\ldots,n+1-j\}\}$ consists of the \textbf{non-free} variables and the
indices $(i,j)$ for $1 \leq i,j \leq n-1, i \in \{2,3,\ldots,n+1-j\}$
give the \textbf{positions of the non-free variables}. The other
variables are the \textbf{free variables}. In particular, observe that $x_{1,1}$ is always a free variable.

We also record the following fact which follows easily from the above
analysis and which we use in Section~\ref{sec:flags}.
\begin{lemma}
  \label{lem: tech lem for constructing a flag of subvarieties}
  Let $h\colon [n] \to [n]$ be an indecomposable Hessenberg function. Then,
  for each pair $(i, j)$ with $n-i\geq j$, we have
  \[
    f^{w_0}_{n+1-i, j} = x_{i+1, j} - g,
  \]
  where $g$ is a polynomial contained in the ideal of
  $\CC [\mathbf{x}_{w_0}]$ generated by
  $\{x_{i,\ell}\mid j+1\leq \ell \leq n-i\}$.
\end{lemma}
\begin{proof}
  Let us denote by $I_{i, j+1}$ the ideal mentioned in the claim.
  From the expression \eqref{eq:5} of $f^{w_0}_{n+1-i,j}$,
  it suffices to show that $y_{i, \ell}\in I_{i, \ell}$ for
  $j+1\leq \ell\leq n-i$.  We fix arbitrary $1\leq i< n$ and
  $j<n-i$, and prove this by induction on $\ell$ with
  $j+1\leq \ell\leq n-i$.  Recall from \eqref{eq:2} with the
  properties \ref{item:1} and \ref{item:2} that we have
  \begin{align}\label{eq: form of yij}
    y_{i, \ell} = - \sum_{k=i}^{n-\ell} y_{i,n+1-k} x_{k,\ell}
    = - x_{i,\ell} - \sum_{k=i+1}^{n-\ell} y_{i,n+1-k} x_{k,\ell},
  \end{align}
  where the second equality follows from $i\leq n-\ell$ and
  $y_{i,n+1-i}=1$.  So when $\ell=n-i$, we have
  \begin{align*}
    y_{i, n-i} = - x_{i,n-i} \in I_{i,n-i}.
  \end{align*}
  Now, by induction, we assume that
  $y_{i,p}\in I_{i,p}\ (\ell+1\leq p\leq n-i)$, and we prove that
  $y_{i, \ell}\in I_{i,\ell}$.  Our polynomial $y_{i,\ell}$ is
  described by the rightmost expression of \eqref{eq: form of
    yij}. There we have $x_{i,\ell}\in I_{i,\ell}$, and also
  $y_{i,n+1-k}\in I_{i,n+1-k}$, by the inductive hypothesis, since
  $\ell+1\leq n+1-k\leq n-i$. These inequalities also imply that we
  have $I_{i,n+1-k}\subset I_{i,\ell}$, and hence we obtain
  $y_{i, \ell}\in I_{i, \ell}$, as desired.
\end{proof}

Having just proved directly that $\CC[\mathbf{x}_{w_0}]/J_{w_0,h}$ is
reduced, the reader may wonder why we do not do the same for all $w
\in \mathfrak{S}_n$. As the proof of Lemma~\ref{pro:2} may suggest, the
argument works out well for $w_0$ due to the particular form of the
matrices $w_0 M$ for $M \in U^- $; for general $w \in
\mathfrak{S}_n$, it seems to be more complicated to analyze these
ideals directly, as the following simple example illustrates.

\begin{example}
Let $n=4$ and $h=(3,3,4,4)$.
Let $w=(2,4,1,3)\in\mathfrak{S}_4$ as in Example~\ref{ex:matrix notation} and Example~\ref{ex:explicit generator}.
The ideal $J_{w,h}$ is generated by $f^w_{4,1}$ and $f^w_{4,2}$ described in \eqref{eq:example of generator}.
  Although one can check computationally (using, say,
  Macaulay2 \cite{M2}) that this ideal is
  reduced, it does not seem so straightforward to prove it directly. 
\end{example}

\begin{proof}[Proof of Propositions~\ref{prop:main claim for Z} and~\ref{prop:defining ideal of Hessenberg}]
As we discussed already, Lemma \ref{lem:CM} and Proposition \ref{pro:2} show that the scheme $\mathcal{Z}(N,h)$ is reduced. This proves Proposition~\ref{prop:main claim for Z}. Moreover, by Lemma \ref{lem:local description}, $\Spec \CC[\mathbf{x}_w]/J_{w,h}$ is reduced 
for each $w\in\mathfrak{S}_n$. That is, we conclude that the ideal $J_{w,h}$ is the defining ideal of $\Hess(N,h)\cap \mathcal{N}_w$. This proves Proposition \ref{prop:defining ideal of Hessenberg}.
\end{proof}

As we saw in Lemma \ref{lem:CM}, the zero scheme $\mathcal{Z}(N,h)$ is a local complete intersection. Now we additionally know that it is reduced, i.e., $\mathcal{Z}(N,h)$ is the integral scheme associated to $\Hess (N,h)$, and thus obtain the following corollary.

\begin{corollary}\label{cor:LCI}
Let $h\colon [n] \to [n]$ be any Hessenberg function. Then the corresponding
regular nilpotent Hessenberg
  variety $\Hess (N,h)$ is a local complete intersection.
\end{corollary}

\begin{proof}
  If the Hessenberg function
  $h$ is indecomposable, the claim holds as we saw above.
  Now suppose that $h$ is not indecomposable. Then by the definition
  of indecomposability we must have 
  $h(j)=j$ for some $j\in \{2,3,\ldots,n-1\}$. In this case,
  $\Hess (N,h)$ is isomorphic to a product of regular
  nilpotent Hessenberg varieties whose Hessenberg functions are
  indecomposable \cite[Theorem 4.5]{Drellich}. Thus the claim holds in this case as well.
\end{proof}

\section{One-parameter families of Hessenberg
  varieties}\label{sec:family}

Let $h\colon [n]\to [n]$ be a Hessenberg function and let
$H(h) \subseteq \mathfrak{gl}_n (\CC)$ be the corresponding Hessenberg
space. The Hessenberg varieties (see Definition~\ref{HessDefn}) with Hessenberg function $h$ can be assembled into a family over
$\mathfrak{gl}_n (\CC)$ defined as

\begin{equation}\label{eq:def-family}
  \{ (MB,X) \in \GL_n (\CC)/B \times \mathfrak{gl}_n (\CC) \mid
  M^{-1} X M \in H(h)\}
  \subseteq 
  \Flags(\CC^n) \times \mathfrak{gl}_n (\CC).
\end{equation}

We are interested in a smaller family which we define as follows.
Throughout the discussion we fix pairwise 
distinct complex numbers $\gamma_1,\gamma_2,\ldots,\gamma_n$.
For $t \in \CC$, we define 
\begin{equation*}
  \Gamma_t :=
  \begin{pmatrix}
    t \gamma_1 & 1 \\
    & t \gamma_2 & 1 \\
    & & \ddots & \ddots \\
    & & & t \gamma_{n-1} & 1\\
    & & & & t \gamma_n
  \end{pmatrix}.
\end{equation*}
Viewing $\CC$ as the complex affine line $\mathbb{A}^1=\mathbb{A}^1_{\CC}$, we define a
family over $\mathbb{A}^1$ by setting
\begin{equation*}
  \X_h := \{(MB, t) \in \Flags (\CC^n) \times \mathbb{A}^1 \mid
  M^{-1} \Gamma_t M \in H(h) \}
\end{equation*}
which can be viewed as a subfamily of~\eqref{eq:def-family} via the
embedding $\mathbb{A}^1 \into \mathfrak{gl}_n(\CC)$ by $t \mapsto
\Gamma_t$, and in particular there is a 
canonical projection
\begin{equation}\label{eq:projection of family} 
  \pi\colon \X_h \longrightarrow \mathbb{A}^1, \qquad
  (MB, t) \longmapsto t.
\end{equation}
The fiber at $t\neq0$ is a regular semisimple Hessenberg variety, and the fiber at $t=0$ is a regular nilpotent Hessenberg variety. In particular, the fibers are irreducible (\cite[Theorem 6 and Corollary 9]{DMPS} and \cite[Lemma 7.1]{AT}). By construction, this morphism is projective, and hence proper. Thus, the irreducibility of the base space $\A^1$ and the fibers imply that the total space $\X_h$ is irreducible as well (cf. the proof of \cite[1. \S~6.3, Theorem 1.26]{Shafarevich}).
In this section, we will prove the following geometric properties of $\X_h$ where we implicitly think of $\X_h$ and $\A^1$ as their associated integral schemes. 

\begin{theorem}
  \label{thm:3}
  Suppose that $h$ is indecomposable. The morphism
  $p \colon \X_h \to \mathbb{A}^1$ is flat, and its scheme-theoretic
  fibers over the closed points of $\mathbb{A}^1$ are reduced.
\end{theorem}

As in Section~\ref{sec:prelim-Hess}, our family $\X_h$ coincides set-theoretically with the zero locus of a section of the vector bundle $(GL_n(\CC) \times_B (\mathfrak{gl}_n(\CC)/H)) \times \mathbb{A}^1$ given by
\[
s_{\Gamma}(MB,t) = ([M,M^{-1}\Gamma_t M],t) 
\]
for $(MB,t)\in \GL_n(\CC)/B\times\CC$. Let $\mathcal{Z}(h)$ be the zero scheme of the section $s_{\Gamma}$ above.
Then we have a morphism $\mathcal{Z}(h)\rightarrow \Spec\CC[t]$ of schemes corresponding to the projection $p\colon \X_h \rightarrow \A^1$. Since $\X_h$ is irreducible as we discussed above, the scheme $\mathcal{Z}(h)$ is irreducible as well.
By \cite[Theorem 6]{DMPS} or \cite[Lemma 7.1]{AT}, the zero locus of $s_{\Gamma}$ in $\GL_n(\CC)/B\times \A^1$ has the expected codimension, namely 
$\sum_{i=1}^n (n-h(i))$. Hence, the zero scheme $\mathcal{Z}(h)$ is Cohen-Macaulay. Thus the morphism $\mathcal{Z}(h)\rightarrow \Spec\CC[t]$ is flat \cite[Section 23]{Matsumura} since the fibers of $\X_h \to \mathbb{A}^1$ have the same dimension.

The product $\Flags (\CC^n) \times \mathbb{A}^1$ is covered by the
affine varieties $\mathcal{N}_w \times \mathbb{A}^1$, for
$w\in \mathfrak{S}_n$ with coordinate ring $\CC[\mathbf{x}_w, t]$. 
The family $\X_h$ is covered by
$\X_h \cap (\mathcal{N}_w \times \mathbb{A}^1)$, for
$w\in \mathfrak{S}_n$,  
and if we define
$F^w_{i,j} := (M^{-1} \Gamma_t M)_{i,j} \in \CC[\mathbf{x}_w, t]$, then
$\X_h \cap (\mathcal{N}_w \times \mathbb{A}^1)$ is set-theoretically cut out by the
equations $F^w_{i,j} = 0$, for all $i,j \in [n]$ with $i> h(j)$.
Let $\mathcal{J}_{w,h} \subseteq \CC[\mathbf{x}_w,t]$ denote the ideal generated by
the $F^w_{i,j}$, for all $i,j \in [n]$ with $i> h(j)$.  One can easily
prove that 
\[
\mathcal{Z}(h)\cap\Spec \CC[\mathbf{x}_w,
t] \cong \Spec \CC[\mathbf{x}_w,
t]/\mathcal{J}_{w,h}.
\]
This gives us an open cover of the scheme $\mathcal{Z}(h)$. In other words, we have 
\[
\mathcal{Z}(h) = \bigcup_{w\in\Sn} \Spec \CC[\mathbf{x}_w,
t]/\mathcal{J}_{w,h}.
\]

We are ready to prove
that the scheme-theoretic fibers of $\mathcal{Z}(h)\rightarrow \Spec\CC[t]$ over the closed points are reduced. For this
purpose, let $z \in \CC$ be 
a closed point in $\Spec\CC[t]$. 
The local ring of $\Spec\CC[t]$ at $z$ is the
localization $\CC [t]_{(t-z)}$. Let
$k(z)$ denote its residue field. Recall that the scheme-theoretic fibre of 
$p\colon \mathcal{Z}(h) \to \Spec\CC[t]$ over $z$ is 
\begin{equation*}
  \mathcal{Z}(h)_z := \mathcal{Z}(h) \times_{\Spec\CC[t]} \Spec \left( k(z) \right).
\end{equation*}
Since $\mathcal{Z}(h)$ is covered by the open affine schemes
$\Spec (\CC [\mathbf{x}_w, t] / \fid)$ for $w\in\mathfrak{S}_n$, 
the fibre $\mathcal{Z}(h)_z$ has an open affine covering consisting of 
\begin{equation*}
  \Spec (\CC [\mathbf{x}_w, t] / \fid) \times_{\Spec\CC[t]} \Spec \left( k(z) \right) \cong \Spec \left( (\CC [\mathbf{x}_w, t] / \fid) \otimes_{\CC [t]} k(z) \right).
\end{equation*}
Consider the ideal $\fid|_{t=z} := \langle F^w_{i,j}|_{t=z} \mid i>h(j)\rangle$
of $\CC [\mathbf{x}_w]$ whose generators are obtained from the
generators of $\fid$ after setting the variable $t$ equal to $z$. Since
the functor $- \otimes_{\CC [t]} k(z)$ has the effect of substituting
$t$ with $z$, we have an isomorphism of rings $  (\CC [\mathbf{x}_w, t] / \fid) \otimes_{\CC [t]} k(z) \cong
  \CC [\mathbf{x}_w] / (\fid|_{t=z})$ and thus 
\begin{equation}\label{eq:10}
  \mathcal{Z}(h)_z = \bigcup_{w\in\mathfrak{S}_n}
  \Spec \left( \CC [\mathbf{x}_w] / (\fid|_{t=z})\right).
\end{equation}
In order to show that the fibres $\mathcal{Z}(h)_z$ are reduced, we will prove
that the rings $\CC [\mathbf{x}_w] / (\fid|_{t=z})$ are reduced.

\begin{proof}[Proof of Theorem~\ref{thm:3}] 
We already saw in the above discussion that the morphism $p \colon \X_h \to \mathbb{A}^1 $ is flat. 
    Now consider the scheme-theoretic fiber at $z=0$.  For any $w\in \mathfrak{S}_n$, we
  have
  \begin{equation*}
    F^w_{i,j}|_{t=0} = (M^{-1} \Gamma_0 M)_{i,j} = (M^{-1} N M)_{i,j} = f^w_{i,j},
  \end{equation*}
  where $f^{w}_{i,j}$ is a generator of the ideal $J_{w,h}$ as
  introduced in Section \ref{sec:ideal}. Then we have an
  equality of ideals $\fid|_{t=0} = J_{w,h}$, for all
  $w\in \mathfrak{S}_n$. It follows that the ring
  $\CC [\mathbf{x}_w] / (\fid|_{t=0})$ is reduced by Proposition
  \ref{prop:defining ideal of Hessenberg}.
  
  Next, consider the case $z\neq0$.
  Focusing on the $w_0$
  patch, the ideal $\fid[w_0]|_{t=z}$ is generated by the polynomials
  $F^{w_0}_{i,j} = (M^{-1} \Gamma_z M)_{i,j}$ with $i> h(j)$. Recall
  that we have $M^{-1} = (y_{i,j})$, with the $y_{i,j}$ satisfying
  equation \eqref{eq:2} and enjoying the properties \ref{item:1},
  \ref{item:2}, and \ref{item:3} recorded in the proof of Proposition
  \ref{pro:2}. For $i<n+1-j$, equation \eqref{eq:2} together with
  properties \ref{item:1} and \ref{item:2} and \ref{item:3} imply that
  \begin{equation*}
    y_{i,j} = - \sum_{k=i+1}^{n-j} y_{i,n+1-k} x_{k,j} - x_{i,j}.
  \end{equation*}
  Hence, by property \ref{item:3}, the polynomial
  \begin{equation}\label{eq:8}
    \widetilde{y}_{i,j} := y_{i,j} + x_{i,j}
  \end{equation}
  does not depend on the variable $x_{i,j}$.

  From the definition of $F^{w_0}_{n+1-i,j}$ it follows that
  \begin{equation*}
    \begin{split}
      F^{w_0}_{n+1-i,j} &=
      \begin{pmatrix}
        0 & \ldots & 0 & 1 & y_{i,n-i} & \ldots & y_{i,1}
      \end{pmatrix}
      \begin{pmatrix}
        z \gamma_1 x_{1,j} + x_{2,j}\\
        \vdots\\
        z \gamma_{n-j-1} x_{n-j-1,j} + x_{n-j,j}\\
        z \gamma_{n-j} x_{n-j,j} + 1\\
        z \gamma_{n+1-j}\\
        0\\
        \vdots\\
        0
      \end{pmatrix}\\
      &= (z \gamma_i x_{i,j} + x_{i+1,j}) + \sum_{k=i+1}^{n-j}
      (z \gamma_k x_{k,j} + x_{k+1,j}) y_{i,n+1-k} +
      z \gamma_{n+1-j} y_{i,j}.
    \end{split}
  \end{equation*}
  Note that the first and last summand always appear because the
  indecomposability of $h$ implies that $i< n+1-h(j) \leq n+1-j$,
  hence $i < n+1-j$.
  The condition $i < n+1-j$ also guarantees that the variable $x_{i,j}$ appearing in the expression above is not 0 or 1 (cf.\ \eqref{eq:11}).
  Using equation \eqref{eq:8}, we obtain
  \begin{equation}\label{eq:9}
  \begin{split}
    F^{w_0}_{n+1-i,j} &= z (\gamma_i - \gamma_{n+1-j}) x_{i,j} + x_{i+1,j} \\ 
    &\qquad\quad
    + \sum_{k=i+1}^{n-j} (z \gamma_k x_{k,j} + x_{k+1,j}) y_{i,n+1-k}
    + z \gamma_{n+1-j} \widetilde{y}_{i,j}.
    \end{split}
  \end{equation}
  The coefficient $z (\gamma_i - \gamma_{n+1-j})$ of $x_{i,j}$ in equation \eqref{eq:9} 
  is nonzero since $z\neq0$ and we assume the $\gamma_k$ are pairwise distinct.
 With the exception of the first term,
  all the terms in equation \eqref{eq:9} depend only on variables
  $x_{k,\ell}$ with $k > i$ and $\ell\geq j$, or $k\geq i$ and $\ell >j$.

  Now a simple inductive argument based on the above observations shows
  that in $\CC [\mathbf{x}_{w_0}] / (\fid[w_0]|_{t=z})$
  the variables $x_{i,j}$ with $1\leq j\leq n-1$ and
  $1\leq i\leq n-h(j)$ can be replaced by expressions involving the
  free variables $x_{i,j}$ with $1\leq j\leq n-1$ and $i>
  n-h(j)$. More formally, we have the following ring isomorphisms
  \begin{equation*}
    \begin{split}
      \CC [\mathbf{x}_{w_0}] / (\fid[w_0]|_{t=z})
      \cong \CC [x_{i,j} \mid 1\leq j\leq n-1,
      i> n-h(j)].
    \end{split}
  \end{equation*}
  It follows that the ring $\CC [\mathbf{x}_{w_0}] / (\fid[w_0]|_{t=z})$ is
  reduced.
\end{proof}

We end this section with an example showing that Theorem
\ref{thm:3} does not hold when $h$ is decomposable.

\begin{example}[Non-reduced fiber when $h$ is decomposable]
  Let $n=2$ and $h=(1,2)$.  We consider the open subset
  $\X_h\cap(\mathcal{N}_{\id} \times \mathbb{A}^1)$ of our
  family $\X_h$ and its fiber at $t=0$.  We have
  $\fid[\id] = \langle F^{\id}_{2,1} \rangle \subseteq
  \CC[x_{1,1},t]$, where
  \begin{equation*}
    F^{\id}_{2,1}=t(\gamma_2-\gamma_1)x_{1,1}-x_{1,1}^2.
  \end{equation*}
  It is easy to see directly that the quotient ring
  $\CC [\mathbf{x}_{\id} , t]/\fid[\id]$ is
  reduced. However, we have $\fid[\id]|_{t=0}=\langle x_{1,1}^2\rangle$. Thus
  the ring $\CC [\mathbf{x}_{\id}]/\langle x_{1,1}^2\rangle$ is not
  reduced. We conclude that scheme-theoretic fiber $(\X_h)_0$ is not
  reduced.
\end{example}

Since $\X_h$ is irreducible and the scheme-theoretic fibers of $\pi\colon \X_h \longrightarrow \mathbb{A}^1$ are reduced over closed points of $\A^1$ when we regard $\X_h$ and $\mathbb{A}^1$ as the associated integral schemes, we obtain the following corollary \cite[\S 1.6]{Fulton IntTh}. 
Let $S$ be the diagonal matrix with eigenvalues $\gamma_1,\ldots,\gamma_n$. Then $\Gamma_z$ for $z\neq0$ is conjugate to $S$ so that we have an isomorphism $\Hess(S,h)\cong \Hess(\Gamma_z,h)$.

\begin{corollary}\label{cor:cohomology class}
Suppose that $h$ is indecomposable.
The regular nilpotent Hessenberg variety $\Hess(N,h)$ and the regular semisimple Hessenberg variety $\Hess(S,h)$ determine the same cycles in $H_*(\GL_n(\CC)/B)$;
\[
[\Hess(N,h)] = [\Hess(S,h)] \quad \text{in } H_*(\GL_n(\CC)/B).
\]
\end{corollary}

\section{Flags of subvarieties in regular nilpotent
  Hessenberg varieties}\label{sec:flags}

The point of this section is to use results and techniques from Section~\ref{sec:ideal} to 
show that, in the case of indecomposable regular nilpotent Hessenberg
varieties, there is a choice of a sequence of (dual) Schubert
varieties which is particularly well-behaved when intersected with
$\Hess(N,h)$. While the construction is interesting in its own right, we were motivated by the theory of Newton-Okounkov bodies. For a given algebraic variety $X$, the computation of Newton-Okounkov bodies associated to
$X$ requires the choice of auxiliary data, one of which is a valuation
on the rational functions on $X$. Natural candidates for such
valuations are given by well-behaved flags of subvarieties of $X$. In
general it is natural to choose such flags which are compatible with
existing structures on $X$.  For instance, for flag varieties $G/B$,
Kaveh showed in \cite{Kaveh-crystal} that flags of Schubert varieties
give rise to Newton-Okounkov bodies with intimate connections to
representation theory. Thus, for Hessenberg varieties, which are
subvarieties of the flag variety $\Flags(\CC^n)$, it is 
natural to consider flags of subvarieties obtained by
intersecting with Schubert varieties, as we discuss here.

Recall from \cite[\S~10.6, p.176]{Fulton2} that the \textbf{dual Schubert variety} $\Omega_w\subseteq \Flags(\CC^n)$ for $w\in\Sn$ is the set of $V_{\bullet}\in \Flags(\CC^n)$ satisfying the condition
\begin{equation*}
\dim(V_p\cap\widetilde{F}_{n-q}) \geq |\{i\leq p \mid w(i)\geq q+1\}|
\end{equation*} 
for $q,p\in[n]$ where $\widetilde{F}_{\bullet}$ is the \textbf{anti-standard flag} given by
$\widetilde{F}_j := \CC e_{n+1-j}\oplus \CC e_{n+2-j}\oplus\cdots\oplus\CC e_{n}$. Recall also that $\codim(\Omega_w\subseteq \Flags(\CC^n))=\ell(w)$ the length of $w\in\Sn$ \cite[\S~10.2, p.159]{Fulton2}.

For a permutation $w\in\Sn$, let us define the
\textbf{rank matrix} $rk(w)$\footnote{This is the notation from \cite{Abe-Billey}. In contrast, \cite{Fulton2} uses $r_{w}(q,p)=rk(w^{-1})[q,p]$.} by
\begin{equation*}
rk(w)[q,p]
:= |\{i\leq p \mid w(i)\leq q\}|.
\end{equation*}
Evidently, $rk(w)[q,p]$ is the rank of the upper left $q \times p$
submatrix of the permutation matrix of $w$. 
Recall that the permutation matrix of $w\in\Sn$ is the matrix which has $1$'s in the $(w(j), j)$-th entries for $1\leq j\leq n$ and $0$'s elsewhere.
For $V_{\bullet}\in \Flags(\CC^n)$, let us consider the composition of the maps 
\[
V_p \hookrightarrow \CC^n \twoheadrightarrow \CC^n/\widetilde{F}_{n-q}.
\]
Then we have 
\begin{equation*}
\rank(V_p \rightarrow \CC^n/\widetilde{F}_{n-q})
= \dim V_p - \dim\ker(V_p \rightarrow \CC^n/\widetilde{F}_{n-q}) = p - \dim(V_p\cap\widetilde{F}_{n-q})
\end{equation*}
and 
\begin{equation*}
rk(w)[q,p]
= |\{i\leq p \mid w(i)\leq q\}|
= p - |\{i\leq p \mid w(i)\geq q+1\}|.
\end{equation*}
Hence, we get
\begin{equation}\label{Flag eq: reformulation of dual Schubert variety}
\Omega_w = \{V_{\bullet} \in \Flags(\CC^n)\mid  \rank(V_p \rightarrow \CC^n/\widetilde{F}_{n-q}) \leq rk(w)[q,p] \text{ for } q,p\in[n] \}.
\end{equation} 

Now, let us write an element $V_{\bullet}\in \Flags(\CC^n)=\GL_n(\CC)/B$ in the standard neighbourhood $\neib{w_0}(\subset \Flags(\CC^n))$ around $w_0B$ by a matrix of the form
\begin{align}\label{Flag local chart w0}
V_{\bullet}=
\begin{pmatrix}
   x_{1,1} &  x_{1,2}  &  \cdots  &  x_{1,n-1}  &  1 \\ 
   x_{2,1} &  x_{2,2}  &  \cdots  &  1 &   \\
   \vdots &  \vdots  &  \iddots  &    \\
   x_{n-1,1} &  1  &        &  \\
  1  &    &    &    
\end{pmatrix}
B.
\end{align}
Then \eqref{Flag eq: reformulation of dual Schubert variety} implies that the opposite 
Schubert variety $\Omega_{w}\cap\neib{w_0}$ (in this neighbourhood) is described as the set of $V_{\bullet}\in \Flags(\CC^n)$ satisfying the condition:
\begin{align*}
\text{the upper-left $q\times p$ matrix in \eqref{Flag local chart w0} has rank at most $rk(w)[q,p]$ for all $q,p\in[n]$. }
\end{align*}

The \textbf{diagram} $D(w)$ of a permutation $w\in\Sn$ is obtained from the matrix of $w^{-1}$ by removing all cells in an $n\times n$ array which are weakly to the right and below a $1$ in $w^{-1}$. The remaining cells form the diagram $D(w)$. Note that the cells of
$D(w^{-1})$ are in bijection with the inversions in $w^{-1}$, and in particular,
the Bruhat length $\ell(w)=\ell(w^{-1})$ of $w$ is equal to $|D(w^{-1})|$. For $w \in \Sn$, we say that the diagram $D(w^{-1})$ \textbf{forms a Young diagram} if
all of the boxes in the diagram are connected. 
From the definitions, the following lemma is immediate.

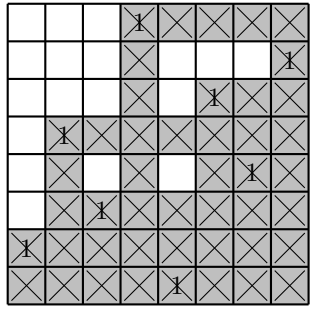
\begin{figure}[h]
  \centering
  \begin{tikzpicture}[scale=0.5]
    \begin{scope}
      \fill[lightgray] (0,0)--(0,2)--(1,2)--(1,5)--(3,5)--
      (3,8)--(8,8)--(8,0)--cycle;
      \fill[white] (2,3)--(3,3)--(3,4)--(2,4)--cycle;
      \fill[white] (4,3)--(5,3)--(5,4)--(4,4)--cycle;
      \fill[white] (4,5)--(5,5)--(5,6)--(7,6)--
      (7,7)--(4,7)--cycle;
      \foreach \i in {0,...,8} {
        \draw[thick] (\i,0)--(\i,8);
        \draw[thick] (0,\i)--(8,\i);
      }
      \foreach \x / \y in {1/2,2/5,3/3,4/8,5/1,6/6,7/4,8/7}
      {
        \draw[shift={(\x,\y)}] (-0.5,-0.5) node {$1$};
        \foreach \u in {\x,...,8}
        {
          \draw[shift={(\u,\y)}] (-0.1,-0.1)--(-0.9,-0.9);
          \draw[shift={(\u,\y)}] (-0.1,-0.9)--(-0.9,-0.1);
        }
      }  
      \foreach \x / \y in {1/1,2/1,3/1,4/1,2/3,2/4,4/4,6/4,4/6,4/7}
      {
        \draw[shift={(\x,\y)}] (-0.1,-0.1)--(-0.9,-0.9);
        \draw[shift={(\x,\y)}] (-0.1,-0.9)--(-0.9,-0.1);
      }  
    \end{scope}
  \end{tikzpicture}

  \caption{For $w=(4,8,6,2,7,3,1,5)$ in one-line notation, $D(w)$ is the
    configuration of white boxes in the array above.}
\end{figure}

\begin{lemma}\label{Flag lem: rk func for YD}
Let $w \in \Sn$ and suppose that $D(w^{-1})$ forms a Young diagram. Then we have
\begin{align*}
 rk(w)[q,p] = 0 \quad \text{for $(q,p)\in D(w^{-1})$}.
\end{align*}
\end{lemma}

\begin{lemma}\label{lemma:dual Schubert intersect Hess}
Suppose that $D(w^{-1})$ forms a Young diagram. Then the opposite 
Schubert variety $\Omega_{w}\cap\neib{w_0}$ (in the affine chart $\neib{w_0}$) is the set of $V_{\bullet}\in \Flags(\CC^n)$ satisfying the condition
\begin{align*}
 x_{q,p} = 0 \quad \text{for $(q,p)\in D(w^{-1})$}
\end{align*}
where $x_{i,j}$ are the coordinates for $\neib{w_0}$ given in \eqref{Flag local chart w0}.
\end{lemma}

\begin{proof}
Let $Z\subseteq \neib{w_0}$ be the (irreducible) 
Zariski closed subset of $V_{\bullet}\in \neib{w_0}(\subset \Flags(\CC^n))$ satisfying 
\begin{align*}
 x_{q,p} = 0 \quad \text{for $(q,p)\in D(w^{-1})$}.
\end{align*}
Then, it is clear from Lemma \ref{Flag lem: rk func for YD} that $\Omega_w\cap\neib{w_0}\subseteq Z$.
Also, we have 
\[
\codim \Omega_w\cap\neib{w_0} = \ell(w) = \ell(w^{-1}) = |D(w^{-1})| = \codim Z
\]
where the first equality uses the fact that $\Omega_w \cap
  \neib{w_0} \neq \emptyset$.
Hence $\dim \Omega_w\cap\neib{w_0} = \dim Z$, and 
since $Z$ is irreducible, we obtain $\Omega_w\cap\neib{w_0}=Z$.
\end{proof}

We now build a flag of subvarieties in
indecomposable regular nilpotent Hessenberg
varieties which looks like a flag of coordinate subspaces near the point
$w_0$. The construction uses a particular sequence of dual Schubert
varieties in $\Flags(\CC^n)$ which we now describe. 
First set 
\[
\dfv := \dim_{\CC} \Flags(\CC^n) = \frac{1}{2}n(n-1)
\]
and let $u_i\in\Sn$ denote the permutation obtained by multiplying
the right-most $i$ simple transpositions of the word
\begin{align*}
(s_1)(s_2 s_1)(s_3 s_2 s_1)\cdots(s_{n-1} s_{n-2} \cdots s_2s_1),
\end{align*}
where $s_i$ denotes the simple transposition exchanging $i$
and $i+1$, and we set $u_0:=\id$. Note that $u_\dfv (=w_0)$ is the longest element.
It is not hard to check that the diagrams 
$D(u_i^{-1})$ form
Young diagrams, and that the Young diagrams corresponding to the sequence 
$u^{-1}_0,
u^{-1}_{1}, \dots, u^{-1}_{\dfv-1}, u^{-1}_\dfv=u_\dfv$ 
``grow'' in sequence by adding boxes
from left to right, starting at the top row.  We
illustrate with an example.

\begin{example} 
Suppose $n=5$. Then 
\begin{align*}
 u_0 &= \hspace{93pt} \id, \\
 u_1 &= \hspace{92pt} s_1, \\
 u_2 &= \hspace{83pt} s_2 s_1, \\
 u_3 &= \hspace{74pt} s_3 s_2 s_1, \\
 u_4 &= \hspace{65pt} s_4 s_3 s_2 s_1, \\
 u_5 &= \hspace{52pt} s_1 \ s_4 s_3 s_2 s_1, \\
 u_6 &= \hspace{43pt} s_2 s_1 \ s_4 s_3 s_2 s_1, \\
 u_7 &= \hspace{34pt} s_3 s_2 s_1 \ s_4 s_3 s_2 s_1, \\
 u_8 &= \hspace{22pt} s_1 \ s_3 s_2 s_1 \ s_4 s_3 s_2 s_1, \\
 u_9 &= \hspace{13pt} s_2 s_1 \ s_3 s_2 s_1 \ s_4 s_3 s_2 s_1, \\
 u_{10} &=  s_1 \ s_2 s_1 \ s_3 s_2 s_1 \ s_4 s_3 s_2 s_1. \\
\end{align*}
The Young diagrams of 
$u_0, u_1, u_2, u_3, u_4, u_5, u_6, u_7, u_8, u_9, u_{10}$
are
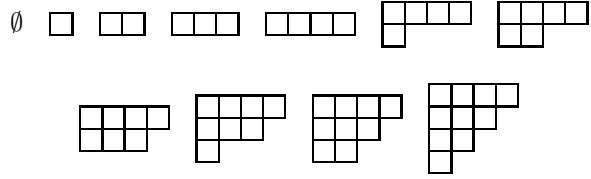
\begin{gather*}
  \emptyset\quad
  \ydiagram{1}\quad
  \ydiagram{2}\quad
  \ydiagram{3}\quad
  \ydiagram{4}\quad
  \ydiagram{4,1}\quad
  \ydiagram{4,2}\\[10pt]
  \ydiagram{4,3}\quad
  \ydiagram{4,3,1}\quad
  \ydiagram{4,3,2}\quad
  \ydiagram{4,3,2,1}
\end{gather*}
\end{example}

We can now define a sequence of subvarieties of $\Hess(N,h)$
by intersecting with a sequence of dual Schubert varieties, as
follows: 
\begin{equation}\label{eq:def flag in Hess}
\begin{split}
\Hess(N,h)=
\Omega_{u_0}\cap\Hess(N,h)
\supseteq
&\Omega_{u_1}\cap\Hess(N,h)
\supseteq \\ 
\cdots
&\supseteq
\Omega_{u_D}\cap\Hess(N,h)
=\{w_0B\}.
\end{split}
\end{equation}
This sequence is not proper in the sense that it may happen that
$\Omega_{u_i} \cap \Hess(N,h) = \Omega_{u_{i+1}}\cap
\Hess(N,h)$ for some $i$. Nevertheless, by omitting
redundancies of the above form, we obtain a flag of subvarieties of
$\Hess(N,h)$ with well-behaved geometric properties within the
open dense subset $\neib{w_0}$. This is
the content of the next theorem and is the main result of this
section. Recall from Section~\ref{sec:ideal} that the defining
equations of $\mathcal{N}_{w_0,h}=\Hess(N,h)\cap\neib{w_0}$ in $\neib{w_0}$ have the
property that some of the coordinates $x_{i,j}$ are free and others are
non-free variables (cf.\ remarks after proof of Lemma~\ref{pro:2}). 

\begin{theorem}\label{theorem:flag of Schuberts}
Let $h\colon  [n] \to [n]$ be an indecomposable Hessenberg function. 
Let $\{u_{\ell}\}_{\ell=0}^D$ be the sequence in $\Sn$ defined above, where
$\dfv=n(n-1)/2$. 
Let $\mathcal{N}_{w_0,h}=\Hess(N,h)\cap\neib{w_0}$
be the open affine chart of $\Hess(N,h)$ around $w_0B$. Then the subvarieties
\begin{equation}\label{eq:flag of Schuberts in Hess}
\mathcal{N}_{w_0,h}=
\Omega_{u_0}\cap\mathcal{N}_{w_0,h}
\supseteq
\Omega_{u_1} \cap \mathcal{N}_{w_0,h}
\supseteq 
\cdots 
\supseteq 
\Omega_{u_D}\cap\mathcal{N}_{w_0,h}
=\{w_0B\}
\end{equation}
satisfy the following:
\begin{enumerate}[label=(\arabic*)]
\item\label{item:4} if the lowest lower-right corner of the Young diagram formed by
  $D(u^{-1}_{\ell})$ is located at the position of a free variable,
  then $\Omega_{u_{\ell-1}} \cap \mathcal{N}_{w_0,h} \neq
  \Omega_{u_{\ell}} \cap \mathcal{N}_{w_0,h}$ and 
\[
\dim \Omega_{u_{\ell}}\cap\mathcal{N}_{w_0,h} = 
\dim \Omega_{u_{\ell-1}}\cap\mathcal{N}_{w_0,h} -1;
\]
 otherwise, $\Omega_{u_{\ell}-1} \cap \mathcal{N}_{w_0,h} =
 \Omega_{u_{\ell}} \cap \mathcal{N}_{w_0,h}$;
\item\label{item:5} each $\Omega_{u_{\ell}} \cap \mathcal{N}_{w_0,h}$ is
  isomorphic to an affine space, and in particular is 
  non-singular and irreducible in $\mathcal{N}_{w_0,h}$.
\end{enumerate}
\end{theorem}

\begin{proof} 
Throughout this argument we use the explicit list of $\dfv=n(n-1)/2$
coordinates on $\neib{w_0} \cong \mathbb{A}^\dfv = \mathbb{A}^{n(n-1)/2}$
given in~\eqref{Flag local chart w0}, totally ordered by reading the
variables from left to right and top to bottom, i.e.\ 
\begin{equation}
  \label{eq:ordered variables}
  x_{1,1}, x_{1,2}, \cdots, x_{1,n-1}, x_{2,1}, x_{2,2}, \cdots,
  x_{2,n-2}, \cdots, x_{n-1,1}.
\end{equation}
Note also that there are exactly as many variables in the list above
as there are elements in the sequence 
\[
u_1, u_2, \cdots, u_\dfv.
\]
As already observed above, from the construction of the sequence $u_\ell$
it follows that the associated diagrams $D(u_{\ell}^{-1})$ form Young
diagrams, and for a given $\ell$, $1 \leq \ell \leq \dfv$, the Young diagram of
$D(u_{\ell}^{-1})$ contains the boxes corresponding to the first $\ell$
variables in the list~\eqref{eq:ordered variables}. We already saw in
Lemma~\ref{lemma:dual Schubert intersect Hess} that $\Omega_{u_{\ell}}
\cap \neib{w_0}$ is equal to the coordinate subspace given by
$\{x_{q,p}=0 \mid (q,p) \in D(u_{\ell}^{-1})\}$, so it follows that the
sequence of intersections $\Omega_{u_{\ell}} \cap \neib{w_0}$ can be
described explicitly in coordinates by setting the first $\ell$ variables
in~\eqref{eq:ordered variables} equal to $0$, i.e.\ we have 
\begin{equation}\label{eq:flag sequence}
\begin{split}
\neib{w_0} \supset \{x_{1,1}=0\} \supset &\{x_{1,1}=x_{1,2}=0\} \supset \\
&\cdots \supset \{x_{1,1}=x_{1,2}=\cdots=x_{n-1,1}=0\} = \{w_0B\}.
\end{split}
\end{equation}
In order to prove the statements in the theorem, we must now also
analyze the intersection of these $\Omega_{u_{\ell}} \cap \neib{w_0}$
with $\Hess(N,h)$. We proceed by induction on $\ell$.

For $\ell =1$, we have
$\CC [\Omega_{u_1} \cap \mathcal{N}_{w_0,h}] \cong \CC
[\mathcal{N}_{w_0,h}] / \langle x_{1,1}\rangle$. As shown in Lemma
\ref{pro:2},
$\CC [\mathcal{N}_{w_0,h}] \cong \CC [\mathbf{x}_{w_0}] / J_{w_0,h}$
is isomorphic to a polynomial ring. Moreover, $D(u^{-1}_1)$ is a
single box located at the position of $x_{1,1}$, which is always a
free variable. Therefore $\CC [\Omega_{u_1} \cap \mathcal{N}_{w_0,h}]$
is isomorphic to a polynomial ring of dimension one less than
$\CC [\Omega_{u_0} \cap \mathcal{N}_{w_0,h}] \cong \CC
[\mathcal{N}_{w_0,h}]$, and $\Omega_{u_1} \cap \mathcal{N}_{w_0,h}$
satisfies properties \ref{item:4} and \ref{item:5}.

For  $\ell >1$, let $x_{i,j}$ denote the $\ell$-th variable in the
ordered list~\eqref{eq:ordered variables}, so that
$\Omega_{u_{\ell}} \cap \mathcal{N}_{w_0,h}$ is obtained from
$\Omega_{u_{\ell-1}} \cap \mathcal{N}_{w_0,h}$ by setting $x_{i,j}$
equal to $0$. (Visually, the position $(i,j)$ is the lowest
lower-right corner of the Young diagram corresponding to
$D(u_{\ell}^{-1})$.) First we consider the case when $x_{i,j}$ is a
free variable. Then it is clear that $x_{i,j}=0$ places
a new linear condition on $\Omega_{u_{\ell-1}} \cap \mathcal{N}_{w_0,h}$.
Moreover, $\CC [\Omega_{u_{\ell-1}}\cap \mathcal{N}_{w_0,h}]$ is irreducible by inductive hypothesis. Therefore the new condition $x_{i,j}=0$ forces
$\Omega_{u_{\ell}} \cap \mathcal{N}_{w_0,h} \neq \Omega_{u_{\ell-1}}
\cap \mathcal{N}_{w_0,h}$ and
$\dim \Omega_{u_{\ell}} \cap \mathcal{N}_{w_0,h} = \dim
\Omega_{u_{\ell-1}} \cap \mathcal{N}_{w_0,h} -1$.  Next suppose that
$x_{i,j}$ is a non-free variable.  As we saw in Lemma~\ref{pro:2}, the
defining equations of $\Hess(N,h)$ within the affine coordinate chart
$\neib{w_0}$ take the form
\[
  x_{i,j} = g
\]
where $x_{i,j}$ is a non-free variable and where $g$ is a polynomial
in the free variables which is contained in the ideal generated by
$x_{i-1,t}$ for $t>j$. Since the sequence~\eqref{eq:flag sequence}
sets variables equal to $0$ in order from left to right and top to
bottom, we know that at this $\ell$-th step, all variables $x_{i-1,t}$
for $t>j$, which are contained in the row directly above that of
$x_{i,j}$, have already been set equal to $0$, and hence $x_{i,j}$ is
already equal to $0$ in
$\Omega_{u_{\ell-1}} \cap \mathcal{N}_{w_0,h}$. Thus the placement of
the additional condition $x_{i,j}=0$ does not affect the intersection
and we conclude that in this case
$\Omega_{u_{\ell}} \cap \mathcal{N}_{w_0,h} = \Omega_{u_{\ell-1}} \cap
\mathcal{N}_{w_0,h}$, as was to be shown.

It follows from the above that 
each $\Omega_{u_{\ell}} \cap \mathcal{N}_{w_0,h}$ is
isomorphic to an affine space with codimension equal to the number of
free variables contained within the first $\ell$ variables in the
sequence~\eqref{eq:ordered variables}. In particular, it is
non-singular and irreducible. This completes the proof.
\end{proof} 

The practical consequence of the above discussion 
is the following. By omitting the redundancies in the
sequence~\eqref{eq:flag of Schuberts in Hess} caused by the non-free
variables, we obtain a flag of subvarieties in $\Hess(N,h)$
(defined in a geometrically natural fashion by intersecting with dual
Schubert varieties) such that, near $w_0B$, the flag is simply a
sequence of affine coordinate subspaces.  It would be interesting to
compute Newton-Okounkov bodies of regular nilpotent Hessenberg
varieties associated to this natural flag. Indeed, the computation of
the special case of
the Peterson variety in Section~\ref{sec:NOBY Peterson} uses
the flag described above.

\section{An efficient formula for the degree of regular nilpotent Hessenberg
  varieties}\label{sec:degree}

Let $\Hess(X,h)$ be a Hessenberg variety in $\Flags(\CC^n)$ and
consider a Pl\"ucker embedding of
$\Flags(\CC^n) \into \P(V_\lambda)$, where $\lambda$ is a strict
partition and $V_\lambda$ is the irreducible representation of
$\GL_n (\CC)$ associated with $\lambda$.  It is then natural to
consider the induced embedding $\Hess(X,h) \into \P(V_\lambda)$, and
to ask for its degree. In this section, we give an efficient computation of the
degree of $\Hess(N,h) \into \mathbb{P}(V_\lambda)$ for all
indecomposable regular nilpotent Hessenberg varieties.  Throughout
this section, we let $S\colon \CC^n \to \CC^n$ be a semisimple
operator with pairwise distinct eigenvalues, and we consider the
associated regular semisimple Hessenberg variety $\Hess (S,h)$.

In Theorem~\ref{thm:3} we showed that a certain family
$\X_h \to \mathbb{A}^1$ of Hessenberg varieties is both flat and
has reduced fibres.  Since Hilbert polynomials are constant along
fibres of a flat family \cite[Theorem 9.9]{Hartshorne} and because the
special fibre is reduced, we can conclude the following. (We can also obtain this result from Corollary~\ref{cor:cohomology class}.) 

\begin{corollary}
  \label{cor:1}
  Let $\lambda$ be a dominant weight and let 
  $\Flags (\CC^n) \hookrightarrow \P(V_\lambda)$ be the
  corresponding Pl\"ucker embedding. By
  composing with the natural inclusion maps, we obtain embeddings
  $\Hess (N,h) \hookrightarrow \P(V_\lambda)$ and
  $\Hess (S,h) \hookrightarrow \P(V_\lambda)$. If $h$ is indecomposable, then
  the degrees of these two embeddings are equal, i.e., 
  \begin{equation*}
    \deg (\Hess(N, h)\hookrightarrow \P(V_\lambda))  = \deg (\Hess(S,h)\hookrightarrow \P(V_\lambda)) .
  \end{equation*}
\end{corollary}

It is known that regular semisimple Hessenberg varieties are smooth
and are equipped with an action of the maximal torus $T$ of
$\GL_n(\CC)$ \cite{DMPS}. In what follows, we use the recent work of Abe, Horiguchi,
Masuda, Murai, and Sato \cite{AHMMS} as well as the classical
Atiyah-Bott-Berline-Vergne formula to obtain a computationally
efficient formula for the degree of the embedding
$\Hess(S,h) \into \P(V_\lambda)$, expressed as a polynomial
in the components of
$\lambda=(\lambda_1>\lambda_2>\dots,\lambda_{n-1}>\lambda_n)$.
By Corollary~\ref{cor:1}, the formula also
computes the degree of $\Hess(N,h)$.

We now turn to the details. Let
$\lambda=(\lambda_1>\lambda_2>\dots,\lambda_{n-1}>\lambda_n) \in
\ZZ^n$ be a strict partition. It is well-known that there is a unique
irreducible representation $V_\lambda$ of $\GL_n(\CC)$ associated with
$\lambda$, and a corresponding Pl\"{u}cker embedding
\[
\Flags(\CC^n)\cong\GL_n(\CC)/B\hookrightarrow \P(V_\lambda)
\]
given by mapping $\Flags(\CC^n)$ to the $\GL_n(\CC)$-orbit
of the highest weight vector in $V_\lambda$.
Composing with the canonical inclusion map $\Hess(N,h)\hookrightarrow \Flags(\CC^n)$, this  
gives us a closed embedding of $\Hess(N,h)$ into $\P(V_\lambda)$. 
Define the \textbf{volume} of this embedding (or of the corresponding
line bundle) by 
\begin{equation}\label{eq:def volume}
\Vol (\Hess(N,h)\hookrightarrow \P(V_\lambda)) := \frac{1}{d!}\deg (\Hess(N,h)\hookrightarrow \P(V_\lambda)) 
\end{equation}
where $d:=\dim_{\CC}\Hess(N,h)=\sum_{j=1}^n (h(j)-j)$.

Using the result from \cite{AHHM} that the cohomology ring
$H^*(\Hess(N,h);\QQ)$ is a Poincar\'e duality algebra generated
by degree 2 elements, the recent work of \cite{AHMMS} relates the
cohomology ring of $\Hess(N,h)$ to other combinatorial and algebraic
invariants; in particular, in \cite[\S~11]{AHMMS} they define,
purely algebraically, a certain polynomial (denoted $P_I$ in \cite[\S~11]{AHMMS})
associated to 
$H^*(\Hess(N,h);\QQ)$. 
The main result of this section is that this polynomial computes the volume
$\Vol (\Hess(N,h)\hookrightarrow \P(V_\lambda))$.
To state the result precisely, we first concretely define the polynomial (up to a
scalar multiple) 
given in \cite{AHMMS} for our special case of Lie type $A_{n-1}$.
Let $\QQ[x_1,\dots, x_n]$ be a polynomial ring in $n$ variables and
for any $i \in [n]$ let $\partial_{x_i}$ denote the usual derivative with
respect to the variable $x_i$. Also for any $i, j \in [n]$ we define $\partial_{i, j}:=\partial_{x_j}-\partial_{x_i}$. 
With this notation in place we may now define, following \cite{AHMMS},
\begin{equation}\label{def:volume polynomial}
P_h(x_1,\dots,x_n) :=
\left(\prod_{h(j)<i}\partial_{i,j}\right) \prod_{1\leq k<\ell\leq n}
  \frac{x_{k}-x_{\ell}}{\ell-k} \in \QQ[x_1,\dots,x_n].
\end{equation}
The theorem below is the main result of this section. 

\begin{theorem}\label{theorem: formula for the degree of nilpotent}
Let $h\colon [n]\to[n]$ be an indecomposable Hessenberg function and let 
$\lambda=(\lambda_1>\lambda_2>\dots>\lambda_n) \in \ZZ^n$ be a strict
partition. Then 
\begin{equation*}
\Vol(\Hess(N,h)\hookrightarrow \P(V_\lambda)) = P_h(\lambda_1,\dots,\lambda_n).
\end{equation*}
\end{theorem}

\begin{proof} 
Consider the regular semisimple Hessenberg variety
$\Hess(S,h)$ corresponding to the same Hessenberg function
$h$ and define the volume $\Vol(\Hess(S,h))$ by the
same formula~\eqref{eq:def volume} (replacing $N$ by
$S$). From the right-hand side of~\eqref{eq:def volume} and by Corollary~\ref{cor:1} it
follows that it suffices to prove that the volume of the regular
semisimple Hessenberg variety is computed by $P_h$, i.e.\ it is enough to show
\[
\Vol(\Hess(S,h)\hookrightarrow \P(V_\lambda)) =
P_h(\lambda_1,\dots,\lambda_n).
\]
Since $\Hess(S,h)$ is non-singular \cite{DMPS}, the degree of a
projective embedding is equal to its symplectic volume \cite[\S~1.3, pg. 171]{GH}:
\begin{equation}\label{eq:volume as integral} 
\Vol (\Hess(S,h)\hookrightarrow \P(V_\lambda)) 
= \frac{1}{d!}\deg (\Hess(S,h)\hookrightarrow \P(V_\lambda)) = \frac{1}{d!}\int_{\Hess(S,h)} c_1(L_{\lambda})^d
\end{equation}
where $c_1(L_{\lambda})$ is the first Chern class of the pullback line
bundle $L_{\lambda}$ on $\Hess(S,h)$ with respect to the
Pl\"ucker embedding and
$d=\dim_{\CC}\Hess(S,h)=\sum_{j=1}^n (h(j)-j)$.  Since the
maximal torus $T$ of $\GL_n(\CC)$ acts on $\Hess(S,h)$
\cite{DMPS}, the Atiyah-Bott-Berline-Vergne
localization formula \cite{at-bo, be-ve} computes this integral 
using the local data around the torus fixed points:
\begin{equation}\label{eq: ABBV computes integral}
\frac{1}{d!}\int_{\Hess(S,h)} c_1(L_{\lambda})^d 
= \frac{1}{d!}\sum_{w\in\Sn} \frac{\lambda_w}{e_w},
\end{equation}
where $\lambda_w$ denotes the weight of the $T$-action on the fiber of
$L_{\lambda}$ at the fixed point $wB$ and $e_w$ denotes the $T$-equivariant
Euler class of the normal bundle to the fixed point $wB$, i.e.,
the product of the weights of the
$T$-representation on the tangent space $T_w\Hess(S,h)$. 

To proceed further we need a more explicit description of the line
bundle $L_{\lambda}$. 
Let $L_i$ denote the $i$-th tautological line bundle over
$\Flags(\CC^n)$, i.e., the fiber of $L_i$ at a flag
$V_{\bullet}\in \Flags(\CC^n)$ is $V_i/V_{i-1}$. Then it is
well-known \cite[\S~9.3]{Fulton2} that 
\begin{align}\label{eq:Plucker line bundle}
L_{\lambda} &\cong 
(L_1^*)^{\lambda_1} \otimes (L_2^*)^{\lambda_{2}} \otimes \cdots \otimes (L_{n}^*)^{\lambda_n} 
\end{align}
is the pullback to $\Flags(\CC^n)$ of $\mathcal{O}(1) \to
\P(V_\lambda)$.
By slight abuse
of notation we also denote by $L_{\lambda}$ this line bundle restricted on $\Hess(S,h)$. 

We can now compute the right-hand side of~\eqref{eq: ABBV computes
  integral}. Recall that the torus $T$ in question is the diagonal
torus $T = \{\operatorname{diag}(t_1, t_2, \dots, t_n) \mid t_i \in \CC^{\times}\}$ of $\GL_n(\CC)$. In
this context, $T$-weights are elements of $\ZZ[t_1,\dots,t_n]$ 
where each 
$t_i$ denotes the weight $T\rightarrow \CC^{\times}$ defined by $\operatorname{diag}(t_1, t_2, \dots, t_n) \mapsto t_i$. 
The weight of the $i$-th
tautological line bundle $L_i$ at the fixed point $w\in\Sn$ 
 is given by $t_{w(i)}$ since the fiber is $\operatorname{span}_{\CC}e_{w(i)}\subset\CC^n$ by definition of $L_i$ where $e_1,\dots,e_n$ are the standard basis of $\CC^n$. 
Thus the weight $\lambda_w$ is
\begin{equation}\label{eq:lambda_w}
\lambda_w 
= - \sum_{i=1}^{n} \lambda_i t_{w(i)}. 
\end{equation}
It is also known \cite{DMPS} that 
the weight $e_w$ is given by
\begin{equation}\label{eq:e_w}
e_w = 
\prod_{j < i \leq h(j)} (t_{w(i)} - t_{w(j)})= (-1)^d\prod_{j < i \leq h(j)} (t_{w(j)} - t_{w(i)}).
\end{equation}
Putting together~\eqref{eq:volume as integral}, ~\eqref{eq: ABBV computes
  integral},~\eqref{eq:lambda_w} and~\eqref{eq:e_w} we therefore obtain 
\begin{equation}\label{eq:volume as polynomial} 
\Vol (\Hess(S,h)\hookrightarrow \P(V_\lambda)) = \frac{1}{d!}\sum_{w\in\Sn} \frac{\left(\sum_{i=1}^{n}\lambda_i t_{w(i)}\right)^{d}}{\prod_{j<i\leq h(j)}(t_{w(j)}-t_{w(i)})}.
\end{equation}
The essential idea of what follows, due to \cite{AHMMS}, is to now
think of the right-hand side of~\eqref{eq:volume as polynomial} as a polynomial in the variables $\lambda_i$. 
More precisely, let us define 
\begin{equation*}
Q_{\Hess(S,h)}(x_1,\dots,x_n) := \frac{1}{d!}\sum_{w\in\Sn} \frac{\left(\sum_{i=1}^{n}x_i t_{w(i)}\right)^{d}}{\prod_{j<i\leq h(j)}(t_{w(j)}-t_{w(i)})}.
\end{equation*}
This is in fact a polynomial in $\R[x_1,\dots,x_n]$ since after taking
the summation over $\Sn$ the right-hand side does not depend on
$t_1,\dots,t_n$ \cite{at-bo, be-ve}.  From the definition it follows that
for any strict partition
$\lambda=(\lambda_1>\lambda_2>\dots>\lambda_n)$ we have
\begin{equation}\label{eq: degree is computed from poly}
\Vol (\Hess(S,h)\hookrightarrow \P(V_\lambda)) = Q_{\Hess(S,h)}(\lambda_1,\dots,\lambda_n).
\end{equation}
Now a straightforward computation shows that
\begin{equation*}
\partial_{i, j}\left(\sum_{i=1}^{n}x_i t_{w(i)} \right)
= t_{w(j)} - t_{w(i)}.
\end{equation*}
From this, it follows from an easy induction argument that 
\begin{equation}\label{eq: degree computation from flag}
Q_{\Hess(S,h)}(x_1,\dots,x_n) = \left(\prod_{\substack{h(j)<i}}\partial_{i,j}\right) Q_{\Flags(\CC^n)}(x_1,\dots,x_n)
\end{equation}
where we think of $\Flags(\CC^n)$ as the regular semisimple
Hessenberg variety with $h=(n,\dots,n)$.  For a strict partition
$\lambda$, the volume of
$\Flags(\CC^n)$ with respect to the Pl\"ucker embedding into
$\P(V_\lambda)$ is well-known to be the 
volume of the Gelfand-Cetlin polytope associated to $\lambda$, for which a
formula is known (e.g.~\cite{ni-no-ue} and \cite{Postnikov}), and we conclude
\begin{equation}\label{eq:volume of flag} 
\Vol(\Flags(\CC^n) \into \P(V_\lambda)) = 
Q_{\Flags(\CC^n)}(x_1,\dots,x_n) = \prod_{1\leq k<\ell\leq n} \frac{x_{k}-x_{\ell}}{\ell-k}.
\end{equation}
From~\eqref{eq: degree computation from flag} and~\eqref{eq:volume of
  flag} we therefore deduce that
\begin{equation}\label{eq:Q equals P_h for semisimple}
Q_{\Hess(S,h)}(x_1,\dots,x_n) =
P_h(x_1,\dots,x_n).
\end{equation}
Thus, from~\eqref{eq: degree is computed from poly} and~\eqref{eq:Q
  equals P_h for semisimple}, we conclude that for a strict
partition $\lambda$
\begin{align*}
\Vol (\Hess(S,h)\hookrightarrow \P(V_\lambda)) = P_h(\lambda_1,\dots,\lambda_n)
\end{align*}
as was to be shown. 
\end{proof}

\begin{remark}\label{rem: lambda_n=0}
Since the line bundle $L_1\otimes\cdots\otimes L_n$ is trivial, we have $L_n\cong L_1^*\otimes\cdots\otimes L_{n-1}^*$.
So we can always assume that $\lambda_n=0$.
\end{remark}

We can use Theorem~\ref{theorem: formula for the degree of nilpotent} to compute the volume of a special
case of a regular nilpotent Hessenberg variety which is studied in
Section~\ref{sec:NOBY Peterson}. 

\begin{example}\label{example:degree of Peterson}
Let $n=3$ and $h=(2,3,3)$, and consider the corresponding regular
nilpotent Hessenberg variety $\Pet_3:=\Hess(N,h)\subset \Flags(\CC^3)$. Then
\begin{align*}
P_h(x_1,x_2,x_3)
&= (\partial_{x_1}-\partial_{x_3}) \left( \frac{(x_1-x_2)(x_1-x_3)(x_2-x_3)}{2}\right) \\
&= \frac{1}{2}(x_1-x_2)^2 + 2(x_1-x_2)(x_2-x_3) + \frac{1}{2}(x_2-x_3)^2.
\end{align*}
So we obtain
\begin{equation*}
\Vol (\Pet_3\hookrightarrow \P(V_\lambda))
= \frac{1}{2}(\lambda_1-\lambda_2)^2 + 2(\lambda_1-\lambda_2)(\lambda_2-\lambda_3) + \frac{1}{2}(\lambda_2-\lambda_3)^2
\end{equation*}
for any strict partition $\lambda=(\lambda_1>\lambda_2>\lambda_3)$.
Let us introduce the notation 
$a_1:=\lambda_2-\lambda_3$ and $a_2:=\lambda_1-\lambda_2$ and set
$\lambda_3=0$ 
following Remark~\ref{rem: lambda_n=0}. Then we have
\begin{equation*}
\Vol (\Pet_3\hookrightarrow \P(V_\lambda))
= \frac{1}{2}a_1^2 + 2a_1a_2 + \frac{1}{2}a_2^2.
\end{equation*}
\end{example}

\section{Newton-Okounkov bodies of Peterson varieties}\label{sec:NOBY Peterson}

The theory of Newton-Okounkov bodies gives a new method of associating
combinatorial data to geometric objects.  In the case of a toric variety $X$, the combinatorics of
its moment map polytope $\Delta$ fully encodes the geometry of $X$,
but this fails in the general case.  Building on the work of Okounkov
\cite{O1,O2}, Kaveh-Khovanskii \cite{KK} and
Lazarsfeld-Musta{\c{t}}{\u{a}} \cite{LazM} construct a convex body
$\Delta$ in $\mathbb{R}^n$ associated to $X$ equipped with the
auxiliary data of a divisor $D$ and a choice of valuation $\nu$ on the
space of rational functions $\CC(X)$.  The theory of Newton-Okounkov
bodies is powerful for several reasons. Firstly, it applies to an
arbitrary projective algebraic variety, and secondly, under a mild
hypothesis on the auxiliary data, the construction guarantees that the
associated convex body $\Delta$ is maximal-dimensional, as in the
classical setting of toric varieties.  Hence one interpretation of the
results of Lazarsfeld-Mustata and Kaveh-Khovanskii is that there
\emph{is} a combinatorial object of `maximal' dimension associated to
$X$, even when $X$ is not a toric variety.  It is
an interesting problem to compute new concrete examples of these bodies, and one of
our motivations for this paper was to compute Newton-Okounkov bodies
of Hessenberg varieties.

In this section we use results of Section~\ref{sec:flags} and~\ref{sec:degree} to give
a concrete computation of
the Newton-Okounkov bodies 
$\Delta(\Pet_3,R(W_\lambda),\nu)$ of the Peterson variety $\Pet_3$, where here $W_\lambda$ is the image of
$H^0(\Flags(\mathbb{C}^3),L_\lambda)$ in
$H^0(\Pet_3,L_\lambda|_{\Pet_3})$ and $L_\lambda$
is the Pl\"{u}cker line bundle over $\Flags(\mathbb{C}^3)$ corresponding to
$\lambda$ (see \cite[\S~9.3]{Fulton2} or \eqref{eq:Plucker line bundle}).  For precise definitions of Newton-Okounkov bodies we refer the reader to e.g. \cite{KK}. We should note that since $\Pet_3$ is a surface, it is already known \cite{LazM, KLM} that the Newton-Okounkov body is a polygon. The question is to determine precisely this polygon; in the present section, we describe it explicitly as a convex hull of a finite number of points. For the purpose of our argument below it is also useful to recall that the volume of $\Delta(\Pet_3,R(W_\lambda),\nu)$ is equal to the degree of $\Pet_3$ (in the appropriate embedding to be recalled below). 

We need some notation. Let
$\lambda = (\lambda_1>\lambda_2>\lambda_3) \in \ZZ^3$ be a  dominant
weight where we may assume without loss of generality that
$\lambda_3=0$. In fact it will be convenient to set the notation
$a_1 := \lambda_2$ and $a_2 := \lambda_1-\lambda_2$ so that
$\lambda = (a_1+a_2, a_1,0)$. Let $L_\lambda$ denote the Pl\"ucker
line bundle obtained from the Pl\"ucker embedding
$\varphi_\lambda \colon \Flags(\CC^3) \to \P(V_\lambda)$ where $V_\lambda$
denotes the irreducible $\GL_3(\CC)$-representation associated with
$\lambda$.  Let $W_\lambda$ denote the image of
$H^0(\Flags(\CC^3), L_\lambda)$ in
$H^0(\Pet_3, L_\lambda \vert_{\Pet_3})$ and let $R(W_\lambda)$ denote
the corresponding graded ring. We use a geometric valuation on
$\Pet_3$ coming from the flag of subvarieties constructed in
Section~\ref{sec:flags}. More specifically, on the affine open chart
$\neib{w_0}$ near the longest permutation
$w_0 = (321) \in \mathfrak{S}_3$, it follows from the analysis in
Section~\ref{sec:ideal} of the defining equations of regular nilpotent
Hessenberg varieties that $\Pet_3^\circ:=\neib{w_0} \cap \Pet_3$ can
be identified with matrices of the form
\begin{equation}\label{eq:coordinates for Pet_3}
\begin{pmatrix} y & x & 1 \\ x & 1 & 0 \\ 1 & 0 & 0 \end{pmatrix} 
\end{equation}
for arbitrary $x, y \in \CC$, and applying Theorem~\ref{theorem:flag
  of Schuberts} in this case, we obtain the flag (restricted to
$\Pet_3^{\circ}$) 
\[
\Pet_3^{\circ} \supset \{x=0\} \supset \{x=y=0\} = \{\pt\}.
\]
Letting $\nu$ denote the valuation corresponding to the above flag,
Theorem~\ref{prop:a2>a1} of this section computes
$\Delta(\Pet_3, R(W_{(a_1+a_2, a_1,0)}), \nu)$ for all values of
$a_1, a_2 \in \ZZ_{>0}$ (we argue separately the cases $a_2 \geq a_1$
and $a_1 \geq a_2$). It is not hard to see that for the usual
lexicographic order on $\ZZ^2$ with $x>y$, the valuation $\nu$ is the
lowest-term valuation.

We briefly recall a well-known
basis for $H^0(\Flags(\CC^3), L_\lambda)$ and compute its
restriction to $\Pet_3^0$ in terms of the variables $x$ and $y$
above. The following discussion is valid for more general flags and
partitions but we restrict to our case for simplicity; see
\cite{Fulton2} for details. 
Let $\mu=(a_1+a_2,a_2,0)$ be the complement partition of $\lambda=(a_1+a_2,a_1,0)$.
For each semistandard Young tableau $T$ of shape $\mu$,
there is an associated section $\sigma_T$ of
$H^0(\Flags(\CC^3),L_\lambda)$ obtained by taking the product of the
Pl\"ucker coordinates corresponding to each column of $T$. We
illustrate with an example. 

\begin{example}
  Let $A$ be a matrix of the form~\eqref{eq:coordinates for Pet_3}
  representing a flag and suppose $T = \ytableaushort{13,2}$. Then the left
  column corresponds to the determinant
  \begin{equation*}
    \det \begin{pmatrix} y & x \\ x & 1 \end{pmatrix}
  \end{equation*}
  of the first and second rows of the left $3 \times 2$ submatrix of
  $A$, while the second column corresponds to the determinant
  $\det (1) = 1$ of the third row of the left $3 \times 1$
  submatrix. Thus $\sigma_T = y-x^2$.
\end{example}

The following is well-known.

\begin{theorem} (\cite[\S~8 and 9]{Fulton2}) 
Let $\lambda = (a_1+a_2, a_1, 0)$ as above. Let $\mu = (a_1+a_2, a_2, 0)$ denote its complement partition. Then
the set $\{\sigma_T\}$ of all sections corresponding to semistandard Young
tableaux of shape $\mu$ is a basis for
$H^0(\Flags(\CC^3),L_\lambda)$.
\end{theorem}

Motivated by the above theorem, for a partition $\lambda = (a_1+a_2, a_1, 0)$, we now analyze the set
$\mathcal{S}_\lambda$ of all semistandard Young tableau of shape
$\mu=(a_1+a_2, a_2, 0)$ with entries in $\{1,2,3\}$. First observe
that, from the definition of $\mu$, our Young tableau contains
columns of length at most $2$. Moreover, since columns must be
strictly increasing, the only possible length-$2$ columns which can appear in
$T \in \mathcal{S}_{\lambda}$ are $\ytableaushort{1,2}, \ytableaushort{1,3}$, and
$\ytableaushort{2,3}$. The only possible length-$1$ columns are $\ytableaushort{1},
\ytableaushort{2}$ and $\ytableaushort{3}$. 
  Moreover, because rows must be weakly
  increasing (from left to right), a column
 $\ytableaushort{1,2}$ must appear to the left of a $\ytableaushort{1,3}$ or a
  $\ytableaushort{2,3}$, and a $\ytableaushort{1,3}$ can only appear to the left of a
  $\ytableaushort{2,3}$, and so on. Thus it is not hard to see that we can uniquely
  represent a semistandard Young tableau of shape  $\mu=(a_1+a_2,
  a_2,0)$ by recording the number of times each type of column
  appears. More formally, let 
\[
k_{12}(T) := \textup{ the number of times the column $\ytableaushort{1,2}$
  appears in $T$}
\]
and 
\[
k_1(T) := \textup{ the number of times the column $\ytableaushort{1}$
  appears in $T$}
\]
and similarly for $k_{13}(T), k_{23}(T), k_{2}(T)$ and $k_{3}(T)$. The
following lemma is straightforward.

\begin{lemma} 
Let $T \in \mathcal{S}_{\lambda}$. Then:
\begin{enumerate}
\item $T$ is completely determined by the 6 integers $k_{12}(T)$,
$k_{13}(T)$, $k_{23}(T)$, $k_1(T)$, $k_2(T)$ and $k_3(T)$;
\item we must have $k_{12}(T) + k_{13}(T) + k_{23}(T) = a_2$, \hspace{1.5mm}
  $k_1(T)+k_2(T)+k_3(T) = a_1$, and if $k_{23}(T) \neq 0$ then
  $k_1(T)=0$. 
\end{enumerate}
Thus the set $\mathcal{S}_\lambda$ is in bijective correspondence with
the set
\begin{equation}
  \label{eq:bijection S_lambda}
  \left\{ (k_{12}, k_{13}, k_{23}, k_1, k_2, k_3) \in \ZZ^6_{\geq 0}
    \left| 
    \begin{matrix}{k_{12} + k_{13}+ k_{23}= a_2}, \\
  {k_1+k_2+k_3= a_1}, \\ k_{23} \neq 0 \Rightarrow
  k_1=0\end{matrix} \right.  \right\}.
\end{equation}
\end{lemma}

\begin{proof}
  By definition, a Young tableau of shape $\mu=(a_1+a_2,a_2,0)$,
  reading left to right, has $a_2$ columns of size 2 and $a_1$ columns
  of size 1. A semistandard Young tableau $T \in
  \mathcal{S}_{\lambda}$ must have 
  weakly increasing rows. Hence the only possible arrangement of the
  length-$2$ columns is to place (starting from the left) all the
  $\ytableaushort{1,2}$'s, then the $\ytableaushort{1,3}$'s, and then the
  $\ytableaushort{2,3}$'s. Since the diagram has $a_2$ many columns of length
  $2$, it is immediate that $k_{12}(T)+k_{13}(T)+k_{23}(T) = a_2$. It
  also follows that the left $a_2$ columns are determined by these 3
  integers. Next consider the length-$1$ columns. 
Again, since rows must be weakly increasing, 
all $\ytableaushort{1}$'s must be placed first, followed by $\ytableaushort{2}$'s,
followed by the $\ytableaushort{3}$'s. 
Finally, if $k_{23}(T)\neq 0$, this means that there is already a $2$
in the top row before reaching the length-$1$ columns, so there cannot
be any $\ytableaushort{1}$'s among the length-$1$ columns, i.e.\ $k_1(T)=0$ as claimed. Again it follows
that these are completely determined by $k_1(T), k_2(T)$ and $k_3(T)$
and that $k_1(T)+k_2(T)+k_3(T)=a_1$. Moreover, it is clear that any 6
positive integers satisfying the conditions of~\eqref{eq:bijection
  S_lambda} correspond to some $T \in \mathcal{S}_{\lambda}$. 
\end{proof}

Based on the above lemma, henceforth we 
specify a semistandard Young tableau $T$ by a tuple of integers
$(k_{12},k_{13}, k_{23}, k_1, k_2, k_3)$ satisfying the
conditions in~\eqref{eq:bijection S_lambda}, and we also use the notation
\begin{equation}\label{fillingnotation}(12)^{k_{12}}(13)^{k_{13}}(23)^{k_{23}}(1)^{k_1}(2)^{k_2}(3)^{k_3}.\end{equation}

\begin{example}
  Suppose $\mu=(5,2,0)$ so that $a_1=3$ and $a_2=2$. The tableau
  $\ytableaushort{11223,33}$ corresponds to $(0,2,0,0,2,1)$ and we also write
  it as 
  $(13)^2(2)^2(3)$. 
\end{example}

We need the following computation.

\begin{lemma}\label{Petpoly}
  Let $T$ be a semistandard Young 
  tableau
  \[T:=(12)^{k_{12}}(13)^{k_{13}}(23)^{k_{23}}(1)^{k_1}(2)^{k_2}(3)^{k_3}\]
as above. 
Then the section $\sigma_T$, restricted to $\Pet_3^{\circ}$ and expressed in
terms of the variables $x$ and $y$ in~\eqref{eq:coordinates for
  Pet_3}, takes the form 
\[
(y-x^2)^{k_{12}}(-x)^{k_{13}}(-1)^{k_{23}}y^{k_1}x^{k_2}1^{k_3}. 
\]
\end{lemma}

\begin{proof}
  Let $A$ denote a $3\times 3$ matrix as in~\eqref{eq:coordinates for
    Pet_3}. By its construction, the section $\sigma_T$ evaluated at
  $A$ takes the form \cite{Fulton2}
  \[(P_{12})^{k_{12}}(P_{13})^{k_{13}}(P_{23})^{k_{23}}(P_1)^{k_1}(P_2)^{k_2}(P_3)^{k_3}\]
where 
\begin{align*}
  &P_{12}=\begin{vmatrix}y&x\\x&1\end{vmatrix}=y-x^2,\quad
  P_{13}=\begin{vmatrix}y&x\\1&0\end{vmatrix}=-x,\quad
  P_{23}=\begin{vmatrix}x&1\\1&0\end{vmatrix}=-1,\quad \\
  &P_{1}=y,\quad
  P_{2}=x,\quad
  P_3=1.
\end{align*}
The result follows. 
\end{proof}

We now compute the Newton-Okounkov bodies $\Delta(\Pet_3,
R(W_{(a_1+a_2,a_1,0)}), \nu)$. Recall from the above discussion (cf. also \cite[Corollary 3.2]{KK}) 
that if we can find
vertices contained in $\Delta(\Pet_3,R(W_{(a_1+a_2,a_1,0)}),\nu)$ whose
convex hull $\Delta$ has volume equal to the degree of $\Pet_3 \into
\P(V_\lambda)$, then
$\Delta=\Delta(\Pet_3,R(W_{(a_1+a_2,a_1,0)}),\nu)$. Since we know the
degree of $\Pet_3 \into \P(V_\lambda)$ from Example~\ref{example:degree of Peterson}, we take this approach in our arguments below.

\begin{theorem}\label{prop:a2>a1}
  Let $\lambda=(a_1+a_2, a_1,0)$ and let $\mu=(a_1+a_2, a_2, 0)$ denote the complement partition of $\lambda$. 
  If $a_1 \geq a_2$, then the
corresponding
  Newton-Okounkov body $\Delta(\Pet_3,R(W_{(a_1+a_2,a_1,0)}),\nu)$
  is the convex hull of the vertices
  \[
\{(0,0), (2a_2+a_1,0), (0,a_1+a_2), (3a_2,a_1-a_2)\}.
\]
If $a_2 \geq a_1$, then the
corresponding
  Newton-Okounkov body $\Delta(\Pet_3,R(W_{(a_1+a_2,a_1,0)}),\nu)$
  is the convex hull of the vertices
  \[
(0,0), (0,a_1+a_2), (2a_1+a_2,0), (3a_1,a_2-a_1).
\]

\end{theorem}

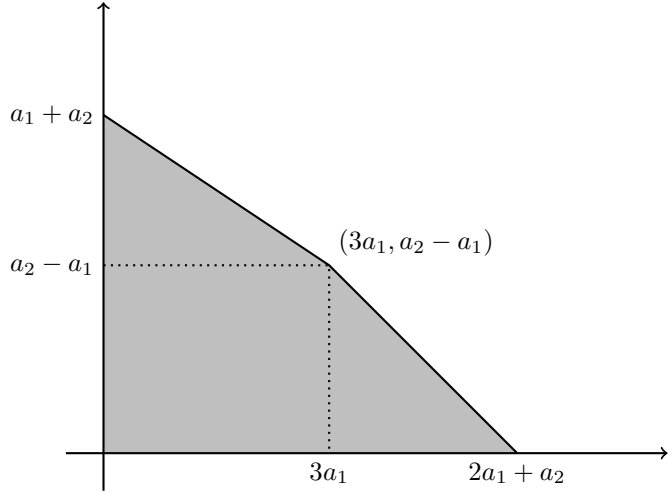
\begin{figure}[h]
  \centering
  \begin{tikzpicture}[scale=0.5]
    \fill[lightgray] (0,0)--(0,9)--(6,5)--(11,0)--cycle;
    \draw[->,thick] (-1,0) -- (15,0);
    \draw[->,thick] (0,-1) -- (0,12);
    \draw[thick,dotted] (0,5)--(6,5)--(6,0);
    \draw[thick] (0,9)--(6,5)--(11,0);
    \draw (0,5) node[anchor=east] {$a_1-a_2$};
    \draw (0,9) node[anchor=east] {$a_1+a_2$};
    \draw (6,0) node[anchor=north] {$3a_2$};
    \draw (11,0) node[anchor=north] {$2a_2+a_1$};
    \draw (6,5) node[anchor=south west] {$(3a_2,a_1-a_2)$};
  \end{tikzpicture}

  \caption{$\Delta(\Pet_3,R(W_{(a_1+a_2,a_1,0)}),\nu)$ for $a_1\geq a_2$.}
\end{figure}

\begin{proof}
We begin with the case $a_1 \geq a_2$. 
First, notice that the area of the convex hull described in the
statement of the theorem 
is
\begin{equation*}
3a_2(a_1-a_2) + \frac{1}{2}(3a_2)(2a_2) + \frac{1}{2}(a_1-a_2)^2 =
\frac{1}{2}a_1^2 + 2a_1a_2 + \frac{1}{2}a_2^2.
\end{equation*}
Therefore, as observed above, it suffices to show that the
four stated vertices
all lie in 
$\nu(W_{(a_1+a_2,a_1,0)})$. We deal with the 
four cases 
separately.

We begin with $(0,0)$.  The semistandard Young tableau
$(23)^{a_2}(3)^{a_1}$ of shape $\mu$ corresponds to the polynomial $1$ (by
Lemma~\ref{Petpoly}), and $\nu(1)=(0,0)$. Hence, $(0,0)$ is in the
image $\nu(W_{(a_1+a_2,a_1,0)})$. 

Next we consider $(0,a_1+a_2)$. The semistandard Young tableau
$(12)^{a_2}(1)^{a_1}$ corresponds to the polynomial
$(y-x^2)^{a_2}y^{a_1}$, and $\nu((y-x^2)^{a_2}y^{a_1})=(0,a_1+a_2)$.

Now we consider $(2a_2+a_1,0)$, for which we look at the set of
tableaux
$(12)^k(13)^{a_2-k}(1)^{a_2-k}(2)^{a_1-a_2+k}$ for $0 \leq k \leq
a_2$. 
Notice that these are valid tableaux because $a_1 \geq a_2$.
By Lemma~\ref{Petpoly} these have corresponding polynomials (up
to sign) 
\begin{align*}g_k:=(y-x^2)^kx^{a_1}y^{a_2-k} &= \left[\sum_{j=0}^k
  (-1)^j\binom{k}{j}y^{k-j}x^{2j}\right]x^{a_1}y^{a_2-k}\\&=\sum_{j=0}^{k}
(-1)^j \binom{k}{j} x^{a_1+2j}y^{a_2-j}.\end{align*} 
Note that the set of $a_2+1$ monomials $x^\alpha y^\beta$ that appear in the
$a_2+1$ polynomials $\left\{g_0,\ldots,g_{a_2}\right\}$ is precisely:
\begin{equation}\label{basisone} 
\left\{x^{a_1}y^{a_2},x^{a_1+2}y^{a_2-1},x^{a_1+4}y^{a_2-2},\ldots,x^{a_1+2a_2}\right\}
\end{equation} 
and also that, with respect to this ordered basis, the
$(a_2+1)\times(a_2+1)$ matrix of coefficients of
the $g_k$ is triangular and invertible. Thus $x^{a_1+2a_2}$ is equal
to 
an appropriate linear
combination of the $g_k$'s and in particular is in
$W_{(a_1+a_2,a_1,0)}$. Since $\nu(x^{a_1+2a_2}) = (a_1+2a_2,0)$ we see
that this vertex lies in the image of $\nu$.

Finally, for the case of the vertex $(3a_2,a_1-a_2)$ we consider the tableaux
$(12)^k(13)^{a_2-k}(1)^{a_1-k}(2)^k$ for $0 \leq k \leq a_2$. Notice
that these are valid tableaux because $a_1 \geq a_2$. 
Again by Lemma~\ref{Petpoly}, we can compute the corresponding
polynomials to be 
\begin{align*}h_k:=(y-x^2)^kx^{a_2}y^{a_1-k}&=\left[\sum_{j=0}^{k}
    (-1)^{j}\binom{k}{j} y^{k-j}
    x^{2j}\right]x^{a_2}y^{a_1-k}\\&=\sum_{j=0}^{k}
  (-1)^{j}\binom{k}{j} x^{a_2+2j}y^{a_1-j}.\end{align*} 
By an argument similar to that above, we can see that there is an
appropriate linear combination of the $h_k$ which equals 
$x^{3a_2}y^{a_1-a_2}$, and since $\nu(x^{3a_2}y^{a_1-a_2}) =
(3a_2,a_1-a_2)$ we conclude that it is in the 
image, as desired.  This concludes the proof for $a_1 \geq a_2$.

Now suppose $a_2 \geq a_1$. We follow the same strategy so we will be brief.
For $(0,0)$ and $(0,a_1+a_2)$ it suffices to consider the tableaux
$(23)^{a_2}(3)^{a_1}$ and $(12)^{a_2}(1)^{a_1}$ respectively. 
For $(2a_1+a_2,0)$, the collection of tableaux of the form
$(12)^k(13)^{a_2-k}(1)^{a_1-k}(2)^{k}$ for varying $k$ as in the proof
of Theorem~\ref{prop:a2>a1} does the job.

For the last case of $(3a_1,a_2-a_1)$ we need the so-called truncated Pascal matrices. 
Recall that an upper-triangular Pascal matrix $T$ is an infinite matrix
  with $(i,j)$-th entry for $i, j \in \ZZ_{\geq 0}$ equal to the binomial coefficient
  $\binom{j-1}{i-1}$, where we take the convention that $\binom{j-1}{i-1}:=0$ if
  $i-1>j-1$. 
  A truncated Pascal matrix is a matrix obtained from an
  upper-triangular Pascal matrix $T$ by
  selecting some arbitrary finite subsets of the rows and columns of T
  of equal size, i.e.
  \[T(r,s):=\begin{pmatrix}
    \binom{s_0}{r_0} & \binom{s_1}{r_0} & \cdots & \binom{s_d}{r_0} \\
    \binom{s_0}{r_1} & \binom{s_1}{r_1} & \cdots & \binom{s_d}{r_1} \\
    \vdots  & \vdots  & \ddots & \vdots \\
    \binom{s_0}{r_d} & \binom{s_1}{r_d} & \cdots & \binom{s_d}{r_d}
    \\ \end{pmatrix},\]
  for some sets $r=\left\{r_0<r_1<\cdots<r_d\right\}$ and
  $s=\left\{s_0<s_1<\cdots<s_d\right\}$, for $s_i,r_i\in\mathbb{N}$.

Now consider the tableaux
$(12)^{a_2-a_1+k}(13)^{a_1-k}(1)^{a_1-k}(2)^{k}$ where
$0 \leq k \leq a_1$. As before we can compute the corresponding
polynomials $h_k$ to be 
\begin{align*}
h_k
      =\sum_{j=0}^{a_2-a_1+k} (-1)^j \binom{a_2-a_1+k}{j}
  x^{a_1+2j}y^{a_2-j}\text{, for } 0\leq k\leq a_1.\end{align*} 
There are $a_2+1$ many monomials $x^\alpha y^\beta$ appearing in these $a_1+1$
polynomials; listed in increasing lex order, they are 
\begin{equation}\label{Petbasis}
\{x^{a_1}y^{a_2}, x^{a_1+2}y^{a_2-1}, x^{a_1+4}y^{a_2-2},\ldots, x^{3a_1}y^{a_2-a_1},\ldots, x^{a_1+2a_2-2}y, x^{a_1+2a_2}\}.
\end{equation}
The $(a_2+1) \times (a_1 +1)$ matrix of coefficients of the $h_k$ with
respect to the ordered basis~\eqref{Petbasis} has $(j, k)$-th entry
equal to $(-1)^j \binom{a_2-a_1+k}{j}$. 

We wish to find a suitable linear combination of the $h_k$ so that its
lowest term is a multiple of $x^{3a_1}y^{a_2-a_1}$. Some elementary
linear algebra shows that it suffices to prove that the upper-left
$(a_1+1) \times (a_1+1)$ submatrix $A$ of the matrix of coefficients
above, with entries equal to $(-1)^j \binom{a_2-a_1+k}{j}$ for $0 \leq
j, k \leq a_1$, is invertible. 
For this it suffices in turn to show that $\det A \neq 0$. Let $A'$
denote the matrix obtained from $A$ by multiplying every other row by
$(-1)$; then $\det A' = \pm \det A$ so it suffices to show $\det A'
\neq 0$. Finally observe that $A'$ is (up to sign) a truncated Pascal
matrix $T(r,s)$ for $r = \{0 < 1 < 2 < \cdots < a_1\}$ and $s = \{
a_2-a_1 < a_2-a_1+1 < \cdots < a_2\}$. By our assumption that $a_2
\geq a_1$ we have that $r_i \leq s_i$ for all $i$. It is known \cite{Kersey} that a 
truncated Pascal matrix is invertible if and only if $r_i \leq s_i$ for all $i$, so we conclude that $\det A' \neq 0$ as
desired. This completes the proof. 
\end{proof}

\end{document}